\documentclass[12pt,leqno]{article}
\usepackage{amssymb,amsfonts,amsmath,amsthm,amscd,mathrsfs}
\def\IR{{\mathbb R}}
\def\IN{{\mathbb N}}
\def\IQ{{\mathbb Q}}
\def\IT{{\mathbb T}}
\def\IZ{{\mathbb Z}}

\def\n{\noindent}
\def\dsl{\textstyle\sum\limits}

\def\dis{\displaystyle}
\def\o{\omega}
\def\fr{\mbox{\footnotesize $\dis\frac{1}{2}$}}

\def\ov{\overline}

\def\ve{\varepsilon}
\def\f{\footnotesize} 
\def\r{\rightarrow}
\def\point{{\mbox{\large $.$}}}
\def\wh{\widehat}
\def\wt{\widetilde}

\def\cA{{\cal A}}

\def\cC{{\cal C}}

\def\cT{{\cal T}}

\def\cI{{\cal I}}

\def\cF{{\cal F}}

\def\cG{{\cal G}}
\def\cO{{\cal O}}

\dimendef\dimen=0

\newtheorem{theorem}{Theorem}[section]
\newtheorem{lemma}[theorem]{Lemma}

\newtheorem{proposition}[theorem]{Proposition}
\newtheorem{remark}[theorem]{Remark}

\begin{document}
\noindent

~
\vspace{1cm}

\bigskip
\begin{center}
{\bf RANDOM WALKS ON DISCRETE CYLINDERS AND  \\RANDOM INTERLACEMENTS}
\end{center}

\begin{center}
Alain-Sol Sznitman\end{center}

\bigskip
\begin{abstract}
We explore some of the connections between the local picture left by the trace of simple random walk on a cylinder $(\IZ / N\IZ)^d \times \IZ$, $d \ge 2$, running for times of order $N^{2d}$ and the model of random interlacements recently introduced in  \cite{Szni07} . In particular we show that for large $N$ in the neighborhood of a point of the cylinder with vertical component of order $N^d$ the complement of the set of points visited by the walk up to times of order $N^{2d}$ is close in distribution to the law of the vacant set of  \cite{Szni07}  with a level which is determined by an independent Brownian local time. The limit behavior of the joint distribution of the local pictures in the neighborhood of finitely many points is also derived.
\end{abstract}

 \vspace{7.5cm}
\n
Departement Mathematik   \\
ETH Z\"urich\\
CH-8092 Z\"urich\\
Switzerland

\newpage

\thispagestyle{empty}
~

\newpage
\setcounter{page}{1}
\setcounter{section}{-1}
\section{Introduction}

The aim of this work is to explore some of the connections between the model of random interlacements recently introduced in \cite{Szni07} and the microscopic structure left by simple random walk on an infinite discrete cylinder with base modelled on a $d$-dimensional torus, $d \ge 2$, of large side-length $N$, when the walk runs for times of order $N^{2d}$. The choice of this specific time scale is motivated by recent works on the disconnection time of the cylinder, where this time scale is shown to govern the magnitude of the disconnection time, at least in principal order, cf.~\cite{DembSzni06}, \cite{DembSzni07}, \cite{Szni08a}. In \cite{Szni07} we introduced the so-called interlacement at level $u \ge 0$, which is the trace left on $\IZ^{d+1}$ by a cloud of paths constituting a Poisson point process on the space of doubly infinite trajectories modulo time-shift, tending to infinity at positive and negative infinite times. The parameter $u$ enters as a multiplicative factor of the intensity measure of this point process. The interlacement at a positive level $u$ is a translation invariant ergodic infinite connected random subset of $\IZ^{d+1}$. Its complement is the so-called vacant set at level $u$. In this work we show that for large $N$ in the neighborhood of a point $x$ of the cylinder with vertical component of order $N^d$, the complement of the trajectory of the walk running up to time $N^{2d}$ is close in distribution to the law of a vacant set with level determined by an independent Brownian local time. The limit behavior of the joint distribution of the local pictures in the neighborhood of finitely many points as above, with mutual distance tending to infinity, is also covered by our results.

\medskip
Before discussing these matters any further, we first present the model more precisely. For $N \ge 1$, $d \ge 2$, we consider the discrete cylinder
\begin{equation}\label{0.1}
E = \IT \times \IZ, \;\mbox{where} \;\IT = (\IZ / N \IZ)^d\,.
\end{equation}

\n
We denote with $P_x$, $x \in E$, resp.~$P$, the canonical law on $E^\IN$ of simple random walk on $E$ starting at $x$, resp. starting with the uniform distribution on $\IT \times \{0\}$, the collection of points at height zero on the cylinder. We write $E_x$ and $E$ for the corresponding expectations, $X_\point$ for the canonical process. Given $x$ in $E$ and $n \ge 0$, the vacant configuration left by the walk in the neighborhood of $x$ at time $n$ is the $\{0,1\}^{\IZ^{d+1}}$-valued random variable:
\begin{equation}\label{0.2}
\o_{x,n}(\cdot) = 1\{X_m \not= \pi_E(\cdot) + x, \;\mbox{for all} \;0 \le m \le n\}\,,
\end{equation}

\n
where $\pi_E$ stands for the canonical projection of $\IZ^{d+1}$ onto $E$. With (\ref{2.16}) of \cite{Szni07}, the law $\IQ_u$ on $\{0,1\}^{\IZ^{d+1}}$ of the indicator function of the vacant set at level $u \ge 0$ is characterized by the property:
\begin{equation}\label{0.3}
\mbox{$\IQ_u(\o(x) = 1$, for all $x \in K) = \exp\{ - u \,{\rm cap}(K)\}$, for all finite sets  $K \subseteq \IZ^{d+1}$}\,,
\end{equation}

\n
where $\o(x)$, $x \in \IZ^{d+1}$, stand for the canonical coordinates on $\{0,1\}^{\IZ^{d+1}}$, and cap$(K)$ for the capacity of $K$, see (\ref{1.7}) below.

\medskip
The local time of the vertical component $Z_\point$ of $X_\point$ is defined as:
\begin{equation}\label{0.4}
L^z_n = \dsl_{0 \le m < n} 1\{Z_m = z\}, \;\mbox{for} \;z\in \IZ, n \ge 0\,.
\end{equation}

\n
We write $W$ for the canonical Wiener measure and $L(v,t)$, $v \in \IR$, $t \ge 0$, for a jointly continuous version of the local time of the canonical Brownian motion. The main result of this work is the following
\begin{theorem}\label{theo0.1}
Consider $M \ge 1$ and for each $N \ge 1$, $x_1,\dots,x_M$ points in $E$ such that
\begin{align}
&\lim\limits_N \;\inf\limits_{1 \le i \not= j \le M} |x_i - x_j| = \infty\,, \label{0.5}
\\[2ex]
&\mbox{$\lim\limits_N \;\dis\frac{z_i}{N^d} = v_i \in \IR$, for $ 1 \le i \le M$, with $z_i$ the $\IZ$-component of $x_i$}\,, \label{0.6}
\end{align}

\n
as well as non-negative integer-valued random variables $T_N$ such that
\begin{equation}\label{0.7}
\mbox{$\dis\frac{T_N}{N^{2d}} \;\underset{N \r \infty}{\longrightarrow} \alpha \in \IR_+$, in $P$-probability}\,.
\end{equation}

\medskip\n
As $N$ tends to infinity the $(\{0,1\}^{\IZ^{d+1}})^M \times \IR^M_+$-valued random variables
\begin{equation}\label{0.8}
\Big(\o_{x_1,T_N},\dots,\o_{x_M,T_N}, \;\dis\frac{L^{z_1}_{T_N}}{N^d} , \dots , \dis\frac{L^{z_M}_{T_N}}{N^d}\Big), \;N \ge 1\,,
\end{equation}

\medskip\n
converge in distribution under  $P$ to the law of the random vector
\begin{equation}\label{0.9}
(\o_1,\dots,\o_M, U_1,\dots,U_M)\,,
\end{equation}

\n
where $(U_1,\dots, U_M)$ is distributed as $\big((d+1) \,L (v_1, \frac{\alpha}{d+1}\big), \dots , (d+1) \,L\big(v_M, \frac{\alpha}{d+1})\big)$, under $W$, and conditionally on $(U_1,\dots,U_M)$ the random variables $\o_i, 1 \le i \le M$, are independent with distribution  $\IQ_{U_i}$.
\end{theorem}

In a loose sense the above theorem states that in the large $N$ limit, given the rescaled local times $L^{z_i}_{T_N} / N^d$, $1 \le i \le M$, the microscopic pictures $\o_{x_i,T_N}$, $1 \le i \le M$, of the complement of the trajectory up to time $T_N$ in the neighborhood of the points $x_i$ are approximately independent and respectively distributed as vacant sets at level $L^{z_i}_{T_N} / N^d$, and these levels have a joint law close to that of the random vector $\big((d+1)$  $L(v_1, \frac{\alpha}{d+1}), \dots, (d+1) \,L(v_M, \frac{\alpha}{d+1})\big)$.

\medskip
One can also derive a statement similar in spirit to Theorem \ref{theo0.1}, which shows that random interlacements at level $u$ describe the limiting microscopic picture left by simple random walk on $(\IZ / N \IZ)^{d+1}$, $d \ge 2$, started with the uniform distribution and run for time $[u N^{d+1}]$, cf.~\cite{Wind08}.

\medskip
Let us further mention that the investigation of the percolative properties of the vacant set of random interlacements was initiated in \cite{Szni07} because of its links with the study of the disconnection time of $E$ by simple random walk, when $N$ is large. It was shown in \cite{Szni07} that for any $d \ge 2$, there is a finite critical value $u_* \ge 0$, such that for $u > u_*$, almost surely all connected components of the vacant set at level $u$ are finite, whereas for $u < u_*$, almost surely there is an infinite connected component in the vacant set at level $u$. The critical value $u_*$ was shown in \cite{Szni07} to be positive, at least when $d \ge 6$, (we recall that $d+1$ here plays the role of $d$ in \cite{Szni07}), and this was recently extended to all $d \ge 2$ in \cite{SzniSido08}. It was also proved in \cite{Teix08} that for any $u$ the number of infinite connected components in the vacant set at level $u$ is almost surely zero or almost surely one.

\medskip
Under the light of Theorem \ref{theo0.1} one can at least in a heuristic way relate the study of the disconnection time of $E$ by simple random walk to the percolative properties of the vacant set of random interlacements at the appropriate levels represented by the last components of the random vectors in (\ref{0.8}), (\ref{0.9}). The rough flavor of the relation is that when disconnection takes place the corresponding local picture for the vacant set should become non-percolative. This heuristic was made precise in \cite{Szni08b} and used to prove the tightness of the disconnection time divided by $N^{2d}$ and find a candidate for the limiting distribution of this random variable in terms of $u_*$ and Brownian local times. Combined with the lower bound of \cite{DembSzni07}, the upper bound of \cite{Szni08b} shows that, at least when $d \ge 17$,  the disconnection time ``lives in scale $N^{2d}$''.

\medskip
We will now give some comments on the proof of Theorem \ref{theo0.1}. The dependence between the various local pictures $\o_{x_i,T_N}$, which persists in the large $N$ limit, creates a difficulty. The strategy employed in the proof is to retain a coarse-grained information on the vertical component of the walk. Given this information we evaluate the distribution of the vacant set left by the walk in the neighborhood of the points $x_i,1 \le i \le M$. To extract this coarse-grained information we introduce grids on $\IZ$, which depend on $N$ and the sequences $z_1,\dots,z_M$, cf.~(\ref{0.6}). These grids enable to define sequences of times corresponding to the successive returns $R_k, k \ge 1$, of the vertical component of the walk to intervals of length of order $d_N$ centered on the grid and departures $D_k, k \ge 1$, from bigger concentric intervals of length of order $h_N$, where
\begin{equation}\label{0.10}
N^d \gg h_N \gg d_N\,.
\end{equation}

 \n
Away from the points $z_1,\dots,z_M$ the grid and the configuration of intervals is periodic with period $a$ suitable multiple of $h_N$. In the neighborhood of the points $z_1,\dots,z_M$ the grid undergoe  s some perturbation. If $w_1 < w_2 < \dots < w_L$ stand for the distinct limiting values $v_i, 1\le i \le M$, in (\ref{0.6}), the construction is such that for large $N$ the blocks $B_\ell = \IT \times I_\ell$, $1 \le \ell \le L$, where $I_\ell$ are intervals centered on the grid with length of order $d_N$, partition the $x_1,\dots,x_M$ according to the different limiting values $w_1,\dots,w_L$ of their rescaled vertical components in (\ref{0.6}). The blocks $B_\ell$ are concentric to bigger blocks $\wt{B}_\ell = \IT \times \wt{I}_\ell$, where $\wt{I}_\ell$, $1 \le \ell \le L$, are pairwise disjoint intervals of length of order $h_N$ with the same mid-point as $I_\ell$. In  addition we assume $h_N \ge N(\log N)^2$, so that between the successive times $R_k,D_k, k \ge 1$, the $\IT$-component of the walk can ``homogenize''.  Let us point out that depending on the sequences $x_1,\dots,x_M$ under consideration the above constraints may force the sequences $h_N$ and $d_N$ to grow almost as fast as $N^d$. 

\medskip
The grids we consider here are different from the regular grids with constant spacing of order $N$ which are used in \cite{DembSzni06}, \cite{DembSzni07}. The role of these grids and of the associated times $R_k, D_k, k \ge 1$, is threefold. They enable to approximately measure the total time elapsed: indeed we are interested in the various local pictures produced at time $T_N$, and in essence one can safely replace $T_N$ with $D_{k_*}$, where $k_*$ is adequately calibrated, cf.~(\ref{2.12}) and below (\ref{4.1}).  They also give a way to construct an approximation to the quantities $L^{z_i}_{T_N} / N^d$, by keeping track of the number of visits of $X_{R_k}$ to the block $B_\ell$ containing $x_i$, for $1 \le i \le M$, see Proposition \ref{prop2.1}. Lastly for $x_i$ in $B_\ell$, they single out which excursions $X_{[R_k,D_k]}$ may affect the vacant set left by the walk in the neighborhood of $x_i$, namely those for which $X_{R_k}$ belongs to $B_\ell$.

\medskip
The proof of Theorem \ref{theo0.1} also uses coupling, see Proposition \ref{prop3.3}. This enables to replace all relevant excursions $X_{[R_k,D_k]}$ by a collection of excursions $\wt{X}_\point^k$, which conditionally on all $Z_{R_m}$, $Z_{D_m}$, $m \ge 1$, are independent and respectively distributed as the walk on the cylinder with a uniformly chosen starting point on $\IT \times \{Z_{R_k}\}$ and conditioned to exit the bigger block containing $\IT \times \{Z_{R_k}\}$ through $\IT \times \{Z_{D_k}\}$. We denote with $P_{Z_{R_k},Z_{D_k}}$ these conditional laws. One also has a good control when $Z_{R_k}$ belongs to $I_\ell$ of the  effect of an excursion under the above conditional law on the local picture near $x_i$ located in $B_\ell = \IT \times I_\ell$, see Lemma \ref{lem4.2}.

\medskip
We will now describe the organization of this article. Section 1 mainly introduces additional notation. 

\medskip
In Section 2 the principal objective is Proposition \ref{prop2.1}, which relates the excursions of the vertical component of the walk attached to the grids under consideration to a measure of the total time elapsed and to an approximation of the rescaled local time of the vertical component at the locations $z_i$, and times of order $N^{2d}$.

\medskip
Section 3 contains the coupling construction, and the main result appears in Proposition \ref{prop3.3}.

\medskip
In Section 4 we provide the proof of Theorem \ref{theo0.1}. An important ingredient is to control the influence on the local picture near $x_i$ in $B_\ell$ of an excursion under the conditional measures $P_{Z_{R_k}, Z_{D_k}}$, with $Z_{R_k}$ in $I_\ell$, see Lemma \ref{lem4.2}.

\medskip
We use the following convention concerning constants. We denote with $c$ or $c^\prime$ positive constants which solely depend on $d$ and $\alpha$, see (\ref{0.7}), with the exception of Section 2 where constants solely depend on the parameters $\gamma$ and $\rho$ of (\ref{2.7}), (\ref{2.11}). Numbered constants such as $c_0, c_1,\dots$ are fixed and refer to their first appearance in the text. Dependence of constants on additional parameters appears in the notation.

\section{Some notation}
\setcounter{equation}{0}

The main object of this section is to introduce additional notation for some of the objects that will be recurrently used in the sequel.

\medskip
We say that two sequences of numbers are limit equivalent if their difference converges to $0$. We write $\pi_{\IT}$ and $\pi_\IZ$ for the canonical projections of $E$ onto $\IT$ and $\IZ$, and denote with $(e_i)_{1 \le i \le d+1}$ the canonical basis of $\IR^{d+1}$. We let $|\cdot |$ and $| \cdot |_\infty$ stand for the Euclidean and $\ell^\infty$-distances on $\IZ^{d+1}$ or for the corresponding distances induced on $E$. The notation $B(x,r)$ stands for the closed $|\cdot |_\infty$-ball with radius $r \ge 0$ and center $x$ in $\IZ^{d+1}$ or $E$. For $A,B$ subsets of $E$ or $\IZ^{d+1}$ we write $A+B$ for the set of elements of the form $x + y$, with $x$ in $A$ and $y$ in $B$, and $d(A,B) = \inf\{ | x-y|_\infty$, $x \in A,y \in B\}$ for the mutual $\ell^\infty$-distance of $A$ and $B$; when $A = \{x\}$ is a singleton, we simply write $d(x,B)$. Given $U$ a subset of $\IZ^{d+1}$ or $E$, we denote with $|U|$ the cardinality of $U$, with $\partial U$ the boundary of $U$ and $\partial_{\rm int} \,U$ the interior boundary of $U$:
\begin{equation}\label{1.1}
\partial U = \{x \in U^c; \exists \,x^\prime \in U, |x - x^\prime| = 1\}, \;\partial_{\rm int} \,U = \{x \in U; \;\exists x^\prime \in U^c, \,|x - x^\prime | = 1\}\,.
\end{equation}

\n
We write $\nu$ for the uniform distribution on $\IT$ and $\nu_z$ for the uniform distribution on $\IT \times \{z\}$, the ``level $z$'' of the cylinder:
\begin{equation}\label{1.2}
\nu = \dis\frac{1}{N^d} \;\dsl_{y \in \IT} \,\delta_y, \quad \nu_z = \dis\frac{1}{N^d} \;\dsl_{x \in \IT \times \{z\}} \,\delta_x, \;\mbox{for $z \in \IZ$}\,.
\end{equation}

\n
The canonical shift on $E^\IN$ is denoted with $(\theta_n)_{n \ge 0}$ and the canonical filtration with $(\cF_n)_{n \ge 0}$. We write $X_\point$ for the canonical process; $Y_\point$ and $Z_\point$ are its respective components on $\IT$ and $\IZ$. Given a subset $U$ of $E$ we denote with $H_U, \wt{H}_U$ and $T_U$ the entrance time in $U$, the hitting time of $U$ and the exit time from $U$:
\begin{equation}\label{1.3}
\begin{split}
H_U & = \inf\{n \ge 0; \, X_n \in U\}, \;\;\wt{H}_U = \inf\{n \ge 1; \,X_n \in U\}\,,
\\[0.5ex]
T_U & = \inf\{n \ge 0; \,X_n \notin U\}\,.
\end{split}
\end{equation}

\n
In case of a singleton $U = \{x\}$, we write $H_x$ or $\wt{H}_x$ for simplicity. We denote with $\tau_k, k \ge 0$, the time of successive displacements of the vertical component $Z_\point$ of $X_\point$:
\begin{equation}\label{1.4}
\tau_0 = 0, \;\tau_1 = \inf\{ n \ge 0; \;Z_n \not= Z_0\}, \;\tau_{k+1} = \tau_1 \circ \theta_{\tau_k} + \tau_k, \;\mbox{for $k \ge 1$}\,.
\end{equation}
Under $P$ the time changed process
\begin{equation}\label{1.5}
\wh{Z}_k = Z_{\tau_k}, \; k \ge 0\,,
\end{equation}

\medskip\n
has the distribution of simple random walk on $\IZ$ starting at the origin.

\medskip
We denote with $P_x^{\IZ^{d+1}}$ the canonical law of simple random walk on $\IZ^{d+1}$ starting at $x$ and with $E_x^{\IZ^{d+1}}$ the corresponding expectation. With an abuse of notation we keep the same notation as in the case of the walk on $E$; in particular we still write $X_\point$ for the canonical process. Given $K$ a finite subset of $\IZ^{d+1}$ and $U \supseteq K$, the equilibrium measure and capacity of $K$ relative to $U$ are defined by
\begin{align}
e_{K,U}(x)  = & \;P_x^{\IZ^{d+1}} [\wt{H}_K > T_U], \;\mbox{for $x \in K$}, \label{1.6}
\\
& \; 0, \;\mbox{for} \; x \notin K\,, \nonumber
\\[2ex]
{\rm cap}_U(K) = & \dsl_{x \in K} \,e_{K,U}(x)  \quad (\le |K|) \,.\label{1.7}
\end{align}

\n
The Green function of the walk killed outside $U$ is
\begin{equation}\label{1.8}
g_U(x,x^\prime) = E_x^{\IZ^{d+1}} \Big[\dsl_{n \ge 0} 1\{X_n = x^\prime , \;n < T_U\}\Big], \; x,x^\prime \in \IZ^{d+1}\,.
\end{equation}

\n
When $U = \IZ^{d+1}$, we drop the subscript $U$ from the notation and simply refer to the corresponding objects as equilibrium measure of $K$, capacity of $K$, or Green function.

\medskip
In the case of the discrete cylinder $E$, when $U \subsetneq E$ is a strict subset of $E$, we define analogously as in (\ref{1.6}) - (\ref{1.8}) the corresponding objects with now $P_x$ and $E_x$ in place of $P_x^{\IZ^{d+1}}$ and $E^{\IZ^{d+1}}_x$.

\medskip
It will sometimes be convenient to consider the continuous time random walks $\ov{X}_\point$, $\ov{Y}_\point$, $\ov{Z}_\point$ on $E, \IT$, and $\IZ$ with respective jump rates $2(d+1)$, $2d$, and $2$. With an abuse of notation we denote with $P_x$, $P_y^\IT$, $P^\IZ_z$ the corresponding canonical laws starting at $x \in E$, $y \in \IT$, $z \in \IZ$. Otherwise we use notation such as $(\ov{\theta}_t)_{t \ge 0}$, $(\ov{\cF}_t)_{t \ge 0}$ or $\ov{H}_U$ to refer to natural continuous time objects. We will also write
\begin{equation}\label{1.9}
\mbox{$\ov{\sigma}_n, n \ge 0$, with $\ov{\sigma}_0 = 0$, for the $n$-th jump time of the canonical process.}
\end{equation}

\n
The continuous time processes are convenient because on the one hand the discrete skeleton
$X_{\ov{\sigma}_n}$, $n \ge 0$, of $\ov{X}_\point$ is distributed as the discrete time walk $X_\point$, and on the other hand for $x = (y,z) \in E$,
\begin{equation}\label{1.10}
\mbox{under $P^\IT_y \times P^\IZ_z, \,(\ov{Y}_{\hspace{-0.5ex}\point}, \ov{Z}_\point)$ has the canonical law $P_x$ governing $\ov{X}_\point$}.
\end{equation}

\n
Note however that the discrete time processes $Y_\point$ and $Z_\point$ are not distributed as the discrete skeletons of $\ov{Y}_{\hspace{-0.5ex}\point}$ and $\ov{Z}_\point$, since they need not jump at each integer time. The next simple lemma will be useful and applies for instance to entrance times or exit times of subsets of $E$ of the form $\IT \times I$ and their compositions under time shifts. We refer to (\ref{1.2}) for the notation.

\begin{lemma}\label{lem1.1} Assume that
\begin{equation}\label{1.11}
\mbox{$X_\tau$ under $P$ has same distribution as $(\ov{Y}_{\hspace{-0.5ex}\ov{\tau}}, \ov{Z}_{\ov{\tau}})$ under $P^\IT_\nu \otimes P_0^\IZ$},
\end{equation}

\n
where $\tau$ is a non-negative integer valued random variable, and $\ov{\tau}$ a non-negative $\sigma(\ov{Z}_\point$)-measurable random variable. Then,
\begin{equation}\label{1.12}
\mbox{$Y_\tau$ is $\nu$-distributed under $P$ and independent of $Z_\tau$}.
\end{equation}
\end{lemma}

\n
\begin{proof}
Denote with $f,g$ functions on $\IT$ and $\IZ$, $g$ bounded. One has
\begin{equation}\label{1.13}
\begin{split}
E[f(Y_\tau)\,g(Z_\tau)] & = E^\IT_\nu \otimes E_0^{\IZ} [f(\ov{Y}_{\ov{\tau}}) \,g(\ov{Z}_{\ov{\tau}})] = E^{\IZ}_0 [g(\ov{Z}_{\ov{\tau}})\,\mu_{\ov{\tau}}(f)] 
\\[0.5ex]
& = E^\IZ_0 [g(\ov{Z}_{\ov{\tau}})] \,\dis\int fd \nu = E[g(Z_\tau)] \,\dis\int f d \nu\,,
\end{split}
\end{equation}

\medskip\n
where for $t \ge 0$, we have set $\mu_t(f) = E_\nu[f(\ov{Y}_t)] = \int f d \nu$, and used the fact that $\nu$ is the stationary distribution of $\ov{Y}_\point$. This proves our claim.
\end{proof}

\section{Some auxiliary results on excursions and local time}
\setcounter{equation}{0}

In this section we introduce partially inhomogeneous grids on $\IZ$ and attach to them successive return and departure times from certain intervals centered on the grids. These times single out excursions of the walk on $\IZ$ under consideration. Our main interest lies in the derivation of approximations in terms of these excursions of the total time elapsed, and of the rescaled local time at locations close to the grid. The main result appears in Proposition \ref{prop2.1}. Throughout this section constants will solely depend on the parameters $\gamma$ and $\rho$ of (\ref{2.7}), (\ref{2.11}) below.

\medskip
We begin with the description of the grids. We are given three sequences of non-negative integers, $(a_N)_{N \ge 1}$, $(h_N)_{N\ge 1}$, $(d_N)_{N \ge 1}$, such that
\begin{equation}\label{2.1}
\begin{array}{rl}
{\rm i)} & \lim\limits_N \,a_N = \lim\limits_N \,h_N = \infty\,,
\\[2ex]
{\rm ii)} &d_N = o(h_N), \;h_N = o(a_N)\,,
\end{array}
\end{equation}

\n
as well as $L$ sequences of points on $\IZ$, $z^*_\ell(N)$, $N \ge 1$, $1 \le \ell \le L$, with $L \ge 1$, some fixed integer. We assume that for large $N$,
\begin{equation}\label{2.2}
\inf\limits_{1 \le \ell \not= \ell^\prime \le L} \;|z^*_\ell(N) - z^*_{\ell^\prime}(N)| \ge 100 h_N\,.
\end{equation}

\medskip\n
We are specifically interested in the case where $a_N = N^d$, but this special choice plays no role for the results of this section. We will see in the next section how we choose the above objects in the context of Theorem \ref{theo0.1}. From now on we implicitly assume that $N$ is large enough so that
\begin{equation}\label{2.3}
20 (d_N + 1) < h_N, \;100 h_N < a_N \;\mbox{and (\ref{2.2}) holds}\,.
\end{equation}

\n
The $N$-th grid is then defined as the disjoint union
\begin{equation}\label{2.4}
\begin{array}{l}
\cG_N = \cG^*_N \cup \cG^0_N, \;\mbox{where} \;\cG^*_N = \{z^*_\ell(N), \;1 \le \ell \le L\} \;\mbox{and}
\\[1ex]
\cG^0_N = \{z \in 2 h_N \,\IZ; \,|z - z^*_\ell(N)| \ge 2 h_N, \;\mbox{for} \;1 \le \ell \le L\}\,.
\end{array}
\end{equation}

\n
For simplicity we will drop the subscript $N$ in what follows and write $z^*_\ell$ in place of $z^*_\ell(N)$. We then introduce the sets
\begin{equation}\label{2.5}
C = \cG + [-d_N, d_N] \subset O = \cG + (-h_N, h_N) \,.
\end{equation}

\n
With (\ref{2.3}), (\ref{2.4}) we see that
\begin{equation}\label{2.6}
\begin{array}{rl}
{\rm i)} & 2 h_N \le |z - z^\prime | < 4 h_N, \;\mbox{for $z,z^\prime$ neighbors in $\cG$},
\\[1ex]
{\rm ii)} & \mbox{the intervals $\wt{I}_z = z + (-h_N, h_N)$, $z \in \cG$, are pairwise disjoint},
\\[1ex]
{\rm iii)} & I_z = z + [-d_N, d_N] \subseteq \wt{I}_z, \; \mbox{for $z \in \cG$} \,.
\end{array}
\end{equation}

\medskip\n
We now turn to the descriptions of the walks on $\IZ$ we consider in this section. We introduce a number
\begin{equation}\label{2.7}
\gamma \in (0,1]\,,
\end{equation}

\medskip\n
and denote with $Q^\gamma_z$, $z \in \IZ$, the canonical law on $\IZ^\IN$ of the random walk on $\IZ$ which jumps to one of its two neighbors with probability $\frac{\gamma}{2}$ and stays at its present location with probability $1 - \gamma$. We write $E^\gamma_z$ for the corresponding expectation. Our main interest lies in the cases $\gamma = (d+1)^{-1}$ and $\gamma = 1$, respectively corresponding to the law of the vertical component of the walk on the discrete cylinder $E$ and to the law of simple random walk on $\IZ$, which governs the process $\widehat{Z}_\point$ of (\ref{1.5}). These specific choices play no special role for the results of this section. We denote with $Z_\point$ the canonical process, but otherwise keep the notation of the previous section concerning the canonical shift, filtration  and the stopping times in (\ref{1.3}). The local time $L^z_n$, $n \ge 0$, $z \in \IZ$, of the canonical process $Z_\point$ is defined as in (\ref{0.4}) and satisfies the additive functional property:
\begin{equation}\label{2.8}
L^z_{n+m} = L^z_n + L^z_m \circ \theta_n, \;\mbox{for}\; n,m \ge  0, \,z \in \IZ\,.
\end{equation}

\n
An important role is played by the systems of excursions corresponding to the successive returns to $C$ and departures from $O$ of the process $Z_\point$
\begin{equation}\label{2.9}
\begin{split}
R_1 & = H_C, \;D_1 = T_O \circ \theta_{R_1} + R_1, \;\mbox{and for $k \ge 1$},
\\[1ex]
R_{k+1} & = R_1 \circ \theta_{D_k} + D_k, \;D_{k+1} = D_1 \circ \theta_{D_k} + D_k\,,
\end{split}
\end{equation}

\medskip\n
so that $0 \le R_1 \le D_1 \le \dots \le R_k \le D_k \le \dots \le \infty$, and all inequalities except maybe the first one are almost surely strict under any $Q^\gamma_z$.

\medskip
It is convenient to consider the quantity
\begin{equation}\label{2.10}
\begin{split}
t_N & = E_0^\gamma [T_{(-h_N + d_N, h_N - d_N)}] + E^\gamma_{d_N} [T_{(-h_N, h_N)}]
\\[1ex]
& = \gamma^{-1} [(h_N - d_N)^2 + h^2_N - d^2_N]\,,
\end{split}
\end{equation}

\medskip\n
which coincides with $E^\gamma_z [D_1]$, for $z \in \partial O(= \cG + \{- h_N, h_N\})$ sufficiently far away from $\cG^*$. We also introduce a parameter 
\begin{equation}\label{2.11}
\rho > 0 \,,
\end{equation}
and define
\begin{equation}\label{2.12}
T = [\rho \,a^2_N], \;k_* = \sigma - [\sigma^{3/4}], \;k^* = \sigma + [\sigma^{3/4}], \;\mbox{where} \;\sigma = \Big[\rho \;\dis\frac{a_N^2}{t_N}\Big]\;,
\end{equation}

\n
(note that with (\ref{2.1}), (\ref{2.10}), $\lim_N \;\sigma = \infty$).

\medskip
The next proposition contains the main results of this section. It enables us to relate the system of excursions introduced in (\ref{2.9}) to the time elapsed or to the local time spent at a point of $C$. The fact that the sequences $h_N/a_N$ and $d_N/h_N$ converge arbitrarily slowly to zero, see (\ref{2.1}), introduces some difficulty in the proof. We recall that (\ref{2.3}) is implicitly assumed.
\begin{proposition}\label{prop2.1}
\begin{align}
&\lim\limits_N \;Q^\gamma_0 \,[D_{k_*} \le T \le D_{k^*}] = 1\,.\label{2.13}
\\[1ex]
&\lim\limits_N \;\sup\limits_{z \in C} \;E^\gamma_0\, [(|L^z_T - L^z_{D_{k_*}} | / a_N) \wedge 1  ] = 0\,.
\label{2.14}
\end{align}
Moreover one has
\begin{equation}\label{2.15}
\sup\limits_N \;\sup\limits_I \;\dis\frac{h_N}{a_N} \;E_0^\gamma \Big[\dsl_{1 \le k \le k_*} \;1\{Z_{R_k} \in I\}\Big] < \infty\,,
\end{equation}
and
\begin{equation}\label{2.16}
\lim\limits_N \;\sup\limits_I \;\sup\limits_{z \in I} \; E^\gamma_0  \Big[\Big| L^z_{D_{k_*}} - \dis\frac{h_N}{\gamma} \; \dsl_{1 \le k \le k_*} \;1\{Z_{R_k} \in I\} \Big| \Big] / a_N = 0\,,
\end{equation}

\medskip\n
where in {\rm (\ref{2.15})} and {\rm (\ref{2.16})} $I$ runs over the collection $u + [-d_N, d_N]$, $u$ in $\cG$, of components of $C$, see {\rm (\ref{2.5})}, and $N$ fulfills {\rm (\ref{2.3})}.
\end{proposition}

\begin{proof}
We begin with the proof of (\ref{2.13}), and first show that:
\begin{equation}\label{2.17}
\lim\limits_N \,Q_0^\gamma\, [T \le D_{k^*}] = 1\,.
\end{equation}

\n
Consider the periodic grid $\wt{\cG} = 2 h_N \IZ$, and the random variable $\wt{D}$ corresponding to $D_1$ in (\ref{2.9}) when one replaces $\cG$ with $\wt{\cG}$ in (\ref{2.5}). Coming back to (\ref{2.10}) we  see that
\begin{equation}\label{2.18}
E^\gamma_{h_N} [\wt{D}_0] =  t_N\,.
\end{equation}

\medskip\n
With the left-hand inequality of (\ref{2.6}) i) and a comparison argument, it follows that
\begin{equation}\label{2.19}
E^\gamma_z [e^{-\lambda D_1}] \le E^\gamma_{h_N} [e^{-\lambda \wt{D}}], \;\mbox{for} \;\lambda \ge 0, z \in \partial O = \cG + \{ - h_N, h_N\}\,.
\end{equation}

\medskip\n
Note that with (\ref{2.3}), (\ref{2.10}), $\sup_{z \in \IZ} \;E^\gamma_z [T_{(- 4 h_N, 4h_N)} ] \le c\,t_N$, so that with Khasminskii's lemma, cf.~\cite{Khas59}, and the strong Markov property, for a suitable positive constant $c_0$,
\begin{equation}\label{2.20}
E^\gamma_{h_N} \Big[ \exp\Big\{\dis\frac{c_0}{t_N} \;\wt{D}\Big\}\Big] \le 2, \;\sup\limits_{z \in \IZ} \;E^\gamma_z \Big[ \exp\Big\{\dis\frac{c_0}{t_N} \;D_1\Big\}\Big] \le 2\,.
\end{equation}

\n
A repeated use of the strong Markov property at times $D_{k^*-1}, \dots, D_1$ and (\ref{2.19}) yield that
\begin{equation}\label{2.21}
E^\gamma_0 [\exp\{ - \lambda D_{k^*}\}] \le E^\gamma_{h_N} [\exp\{ - \lambda \wt{D}\}]^{(k^* - 1)}\,,
\end{equation}
and therefore we find that
\begin{equation}\label{2.22}
Q^\gamma_0 [D_{k^*} \le T] \le \exp\{\lambda(T - (k^* - 1) \,t_N)\}\;E^\gamma_{h_N} [e^{-\lambda \wh{D}}]^{(k^*-1)}\,,
\end{equation}

\n
where $\wh{D} = \wt{D} - E_{h_N}^\gamma [ \wt{D}] \stackrel{(\ref{2.18})}{=} \wt{D} - t_N$. Moreover we also have:
\begin{equation}\label{2.23}
\begin{split}
E^\gamma_0 \Big[\exp\Big\{\dis\frac{u}{t_N} \;\wh{D}\Big\}\Big] & = 1 + \dis\frac{u}{t_N} \;E^\gamma_{h_N} [\wh{D}] + u^2 \;\dis\int^1_0 \,ds \;\dis\int^s_0 \,dt \;E_{h_N}^\gamma \Big[\Big( \dis\frac{\wh{D}}{t_N}\Big)^2 \,e^{t \,\frac{u}{t_N} \,\wh{D}}\Big]
\\[1ex]
& \le 1 + c_1 \,u^2, \;\mbox{if} \;|u| \le c_2\,,
\end{split}
\end{equation}

\n
where we used the left-hand inequality of (\ref{2.20}) to bound the last term of (\ref{2.23}) and the fact that the $Q^\gamma_{h_N}$-expectation of $\wh{D}$ vanishes. Choosing $\lambda = u/t_N$ in (\ref{2.22}) with $0 < u \le c_2$, we find that
\begin{equation}\label{2.24}
\begin{split}
Q^\gamma_0 [D_{k^*} \le T] &\le  \exp\Big\{ \dis\frac{u}{t_N} \;(T - (k^* -1 ) \,t_N)\Big\} (1 + c_1 \,u^2)^{k^*} 
\\ 
&\hspace{-1ex} \stackrel{(\ref{2.12})}{\le} \exp\{c - u [\sigma^{3/4}] + 2 c_1 \,u^2 \sigma\}\,.
\end{split}
\end{equation}

\medskip\n
Choosing $u = \frac{1}{4c_1} \;\sigma^{-1/4}$, we see that for large $N$
\begin{equation}\label{2.25}
Q^\gamma_0 [D_{k^*} \le T] \le c \,\exp\{-c \,\sigma^{1/2}\}\,,
\end{equation}

\n
which tends to zero, see below (\ref{2.12}), and this proves (\ref{2.17}). We will now prove that
\begin{equation}\label{2.26}
\lim\limits_N \;Q_0^\gamma[D_{k_*} \le T] = 1\,.
\end{equation}

\n
With (\ref{2.6}) i) and (\ref{2.9}) we see that when $z \in \partial O$ is at distance $d(z,\cG^*)$ at least $5 h_N$ from $\cG^*$, then
\begin{equation}\label{2.27}
\mbox{the distribution of $D_1$ under $Q^\gamma_z$ coincides with that of $\wt{D}$ under $Q^\gamma_{h_N}$.}
\end{equation}
We thus can write (see above (\ref{1.1}) for the notation):
\begin{equation}\label{2.28}
\begin{split}
D_{k_*} &= D_1 + \Sigma_1 + \Sigma_2, \;\mbox{where}
\\[1ex]
\Sigma_1& = \dsl_{1 \le k < k_*} D_1 \circ \theta_{D_k} 1\{d (Z_{D_k}, \cG^*) \ge 5 h_N\}\,, 
\\[1ex]
\Sigma_2 &= \dsl_{1 \le k < k_*} D_1 \circ \theta_{D_k} 1\{d(Z_{D_k},\cG^*) < 5 h_N\}\,,
\end{split}
\end{equation}

\n
and $\Sigma_2$ is the term where the inhomogeneity of the grid is mostly felt. With (\ref{2.27}) and the application of the  strong Markov property at times $D_{k_*-1},\dots,D_1$, we find that:
\begin{equation}\label{2.29}
E^\gamma_0 [\exp\{\lambda \Sigma_1\}] \le E^\gamma_{h_N} [\exp\{\lambda \wt{D}\}]^{k_*}, \;\mbox{for} \;\lambda \ge 0\,,
\end{equation}

\n
where of course both members may be infinite. With a similar argument as in (\ref{2.24}), we see that for $0 < u \le c_2$,
\begin{equation}\label{2.30}
Q^\gamma_0 [\Sigma_1 \ge T - t_N \,\sigma^{5/8}] \le \exp \Big\{- \dis\frac{u}{t_N} (T - t_N \,\sigma^{5/8} - k_* \,t_N) + c_1 \,u^2 \,k_*\Big\}
\end{equation}

\n
from which one deduces that for large $N$
\begin{equation}\label{2.31}
Q^\gamma_0 [\Sigma_1 \ge T-t_N \,\sigma^{5/8}] \le c \,\exp\{ - c \,\sigma^{1/2}\}
\end{equation}

\n
which tends to $0$ as $N$ tends to infinity.

\medskip
With the second inequality of (\ref{2.20}), it straightforwardly follows that
\begin{equation}\label{2.32}
\lim\limits_N \;Q^\gamma_0 \Big[D_1 \ge \fr \; t_N \;\sigma^{5/8}\Big] = 0\,.
\end{equation}

\n
The claim (\ref{2.26}) will thus follow once we show that
\begin{equation}\label{2.33}
\lim\limits_N \,Q^\gamma_0 \Big[\Sigma_2 \ge \fr \;t_N \,\sigma^{5/8}\Big] = 0\,.
\end{equation}
For this purpose we will use the following
\begin{lemma}\label{lem2.2}
\begin{equation}\label{2.40}
\sup\limits_{z,z^\prime \in \IZ} \;E^\gamma_{z^\prime} [L^z_n] \le c \,\sqrt{n}, \;\;\mbox{for $n \ge 0$},
\end{equation}

\medskip\n
and with the same notation as in {\rm (\ref{2.15})}, we also have:
\begin{equation}\label{2.41}
\sup\limits_N \;\sup\limits_I \;\dis\frac{h_N}{a_N} \;E^\gamma_0 \Big[\dsl_{1 \le k < k_*} 1\{d(Z_{D_k}, I) \le 5 h_N\}\Big] \le c < \infty\,.
\end{equation}
\end{lemma}

\begin{proof}
We begin with the proof of (\ref{2.40}). General bounds on heat kernels of random walks, see for instance Corollary 14.6 of \cite{Woes00} imply that
\begin{equation}\label{2.42}
\sup\limits_{z,z^\prime \in \IZ} \,Q^\gamma_{z^\prime} [Z_n = z] \le c\,n^{-1/2}, \;\mbox{for} \;n \ge 1\,.
\end{equation}

\n
It now follows that with (\ref{0.4})
\begin{equation*}
\begin{split}
E^\gamma_{z^\prime} [L^z_n] & = \dsl_{0 \le m < n} Q^\gamma_{z^\prime} [Z_m = z] \le 1 + c \dsl_{1 \le m < n} m^{-1/2} \le c\,n^{1/2}, \;\mbox{if} \;n \ge 1\,,
\\[1ex]
& = 0, \;\mbox{if} \;n = 0\,.
\end{split}
\end{equation*}
The claim (\ref{2.40}) follows.

\medskip
Let us now prove (\ref{2.41}). Using the strong Markov property and (\ref{2.20}) we find choosing $c_3 = 2\,c_0^{-1}$ that:
\begin{equation}\label{2.43}
Q_0^\gamma[D_{k_*} \ge c_3 \,T] \le \exp\Big\{- c_0\,c_3\;\dis\frac{T}{t_N} + (\log 2) \,k_*\Big\} \stackrel{\rm (\ref{2.12})}{\le} c\, \exp\{ - c\,\sigma\}\,.
\end{equation}
We thus see that
\begin{equation}\label{2.44}
\begin{array}{l}
\dsl_{1 \le k < k_*} Q_0^\gamma [d(Z_{D_k}, I) \le 5 h_N] \le 
\\[3ex]
\quad k_* \,Q_0^\gamma [D_{k_*} \ge c_3\,T]  +
\dsl_{1 \le k < k_*} Q_0^\gamma [D_k < c_3 \,T, d(Z_{D_k},I) \le 5 h_N] \,.
\end{array}
\end{equation}

\medskip\n
With (\ref{2.12}) and (\ref{2.43}) we see that the first term in the right-hand side of (\ref{2.44}) is bounded by a constant. As for the second term, note that with (\ref{2.19}), calculating derivatives in $\lambda = 0$, we have:
\begin{equation}\label{2.45}
E^\gamma_{h_N} [\wt{D}] = t_N \le E_z^\gamma [D_1], \;\mbox{for} \;z \in \partial O\,.
\end{equation}

\medskip\n
As a result for $k \ge 1$, with the strong Markov property applied at time $D_k$, we find that
\begin{equation}\label{2.46}
t_N \,Q^\gamma_0 [D_k < c_3\,T,\, d(Z_{D_k},I) \le 5 h_N] \le E^\gamma_0 [D_1 \circ \theta_{D_k}, D_k < c_3 \,T, d(Z_{D_k}, I) \le 5 h_N]\,.
\end{equation}

\medskip\n
It now follows that the last term of (\ref{2.44}) is smaller than
\begin{equation*}
t_N^{-1} \,E^\gamma_0 \Big[\dsl_{1 \le k < k_*} D_1 \circ \theta_{D_k} 1\{D_k < c_3 \,T, d(Z_{D_k},I) \le 5 h_N\}\Big]\,.
\end{equation*}

\n
Observe that on the event $\{d(Z_{D_k},I) \le 5 h_N\}$, $Q^\gamma_0$-a.s. for all $n$ in $[D_k,D_{k+1}]$, $d(Z_n,I) \le 10 h_N$, and the above expression is smaller than
\begin{equation}\label{2.47}
\begin{array}{l}
t^{-1}_N \,E^\gamma_0 \Big[\dsl_{0 \le n < [c_3 \,T]} \,1\{d(Z_n, I) \le 10 h_N\} + D_1 \circ \theta_{[c_3 \,T]}\Big] \le
\\[2ex]
t^{-1}_N \Big(\dsl_{z:d(z,I) \le 10h_N} E^\gamma_0 [L^z_{[c_3 \,T]}] + \sup\limits_{z \in \IZ} \,E_z^\gamma[D_1]\Big) \stackrel{\rm (\ref{2.20}),(\ref{2.40})}{\le} c \Big(\dis\frac{h_N}{t_N} \;\sqrt{T} + 1\Big)
\\
\\[-1ex]
 \stackrel{\rm (\ref{2.10}),(\ref{2.12})}{\le} c\Big(\dis\frac{a_N}{h_N} + 1\Big) \,.
\end{array}
\end{equation}

\medskip\n
Collecting the bounds on the right-hand side of (\ref{2.44}), with (\ref{2.3}), our claim (\ref{2.41}) readily follows.
\end{proof}
 
We can now bound the expectation of $\Sigma_2$ in (\ref{2.28}) as follows. With the strong Markov property applied at time $D_k$ and (\ref{2.20}) we see that
\begin{equation}\label{2.48}
\begin{split}
E^\gamma_0 [\Sigma_2] & \le \dsl_{1 \le k < k^*} c\,t_N \;Q_0^\gamma[d(Z_{D_k}, \cG^*) \le 5 h_N]
\\[2ex]
\stackrel{\rm (\ref{2.41})}{\le} & c\,L \,t_N \;\dis\frac{a_N}{h_N} \stackrel{\rm (\ref{2.10}),(\ref{2.12})}{\le} c\,L \,t_N \,\sigma^{1/2}, \;\mbox{for large $N$}\,.
\end{split}
\end{equation}

\n
The claim (\ref{2.33}) immediately follows from Chebyshev's inequality. This concludes the proof of (\ref{2.13}).

\medskip
We now turn to the proof of (\ref{2.14}). Note that with the strong Markov property applied at time $D_{k_*}$ we have
\begin{equation}\label{2.49}
\underset{N}{\overline{\lim}} \,Q^\gamma_0 [D_{k^*} \ge D_{k_*} + c_3 \,t_N(k^* - k_*)] \le \underset{N}{\overline{\lim}} \;\sup\limits_{z \in \IZ} \;Q^\gamma_z [D_{k^* - k_*} \ge c_3 \,t_N (k^* - k_*)] = 0\,,
\end{equation}

\n
where we used a bound similar to (\ref{2.43}) in the last step. With (\ref{2.13}) we thus find that
\begin{equation}\label{2.50}
\begin{array}{l}
\underset{N}{\overline{\lim}} \;\sup\limits_{z \in \IZ} \;E^\gamma_0 [ | (L^z_T - L^z_{D_{k^*}}) / a_N | \wedge 1] \le
\\[-1ex]
\underset{N}{\overline{\lim}} \;\sup\limits_{z \in \IZ}\;E^\gamma_0 \, [ |(L^z_{D_{k_*} + [c_3 \,t_N (k^* - k_*)]} - L^z_{D_{k_*}}) /a_N | \wedge 1] \stackrel{\overset{\mbox{\scriptsize (\ref{2.8})}}{\rm strong\; Markov}}{\le}
\\[2ex]
\underset{N}{\overline{\lim}} \;\sup\limits_{z,z^\prime \in \IZ} \;E^\gamma_{z^\prime} [L^z_{[c_3 \,t_N (k^* - k_*)]} / a_N] \stackrel{\rm (\ref{2.40})}{\le} \underset{N}{\overline{\lim}} \; \;c \,\sqrt{t}_N \;\sigma^{3/8} \,a_N^{-1} \stackrel{\rm (\ref{2.10}),(\ref{2.12})}{=} 0\,,
\end{array}
\end{equation}
which proves (\ref{2.14}).

\medskip
The claim (\ref{2.15}) immediately follows from (\ref{2.41}), once one notes that $Q^\gamma_0$-a.s. on the event  $\{Z_{R_k} \in I\}$ one has $d(Z_{D_k},I) = h_N$.

\medskip
There remains to prove (\ref{2.16}). Given $I = [z_-,z_+]$ as in (\ref{2.16}), so that  $z_0 = \frac{1}{2} \;(z_+ + z_-) \in \cG$, $\wt{I} = (z_0 - h_N, z_0 + h_N)$, and $z \in I$, we define
\begin{equation}\label{2.51}
\begin{split}
M_k & = L^z_{D_k} - \dsl_{1 \le k^\prime \le k} 1\{Z_{R_{k^\prime}} \in I\} \;E^\gamma_{Z_{R_{k^\prime}}} [L^z_{T_{\wt{I}}}], \;\mbox{for} \;k \ge 1\,, 
\\
& = 0, \;\mbox{for $k = 0$} \,.
\end{split}
\end{equation}
Observe that
\begin{equation}\label{2.52}
\mbox{$(M_k)_{k \ge 0}$ is an $(\cF_{D_k})_{k \ge 0}$-martingale under $Q^\gamma_0$.}
\end{equation}

\medskip\n
Indeed $M_k$ is $\cF_{D_k}$-measurable and
\begin{equation*}
\begin{array}{l}
E_0^\gamma \big[M_{k+1} - M_k \,| \,\cF_{D_k}\big] =
\\[2ex]
E_0^\gamma \big[L^z_{D_1} \circ \theta_{D_k} - (1 \{Z_{R_1} \in I\} \;E^\gamma_{Z_{R_1}}[L^z_{T_{\wt{I}}}]) \circ \theta_{D_k} \,|\,\cF_{D_k}\big] =
\\[2ex]
E^\gamma_{Z_{D_k}} \big[L^z_{D_1} - 1\{Z_{R_1} \in I\} \;E^\gamma_{Z_{R_1}}[L^z_{T_{\wt{I}}}]\big] = 0\,,
\end{array}
\end{equation*}

\medskip\n
using the strong Markov property and the $Q^\gamma_{z^\prime}$-a.s. identity $L^z_{D_1} = 1\{Z_{R_1} \in I\}\,L^z_{T_{\wt{I}}} \circ \theta_{R_1}$, for any $z^\prime \in \IZ$, in the last step. As a result we also see  that:
\begin{equation}\label{2.53}
\begin{split}
E^\gamma_0 [M_{k_*}^2]  & = \dsl_{0 \le k < k_*} \;E^\gamma_0 [(M_{k+1} - M_k)^2] 
\\ 
& = \dsl_{0 \le k < k_*} \,E^\gamma_0 \big[\big(1 \{Z_{R_1} \in I\} (L^z_{T_{\wt{I}}} \circ \theta_{R_1} - E^\gamma_{Z_{R_1}} [L^z_{T_{\wt{I}}}]\big)^2\big) \circ \theta_{D_k}\big]\,.
\end{split}
\end{equation}

\medskip\n
Note that for $z^\prime$ in $I$ one has the identities, see above (\ref{2.51}) for the notation
\begin{equation}\label{2.54}
\begin{array}{l}
E^\gamma_{z^\prime} \,[L^z_{T_{\wt{I}}}] = Q^\gamma_{z^\prime} [H_z < T_{\wt{I}}] \;E^\gamma_z [L^z_{T_{\wt{I}}}], \;\mbox{with} \;|Q^\gamma_{z^\prime} [H_z < T_{\wt{I}}] - 1| \le c \;\dis\frac{d_N}{h_N} \; \;\mbox{and} 
\\[1ex]
E^\gamma_z \,[L^z_{T_{\wt{I}}}] = \mbox{\f $\dis\frac{2}{\gamma}$} \;[(h_N - z + z_0)^{-1} + (h_N + z - z_0)^{-1}]^{-1}\,.
\end{array}
\end{equation}

\medskip\n
In particular  we see that for $z,z^\prime$ in $I$:
\begin{equation}\label{2.55}
E^\gamma_{z^\prime} \,[L^z_{T_{\wt{I}}}] = \gamma^{-1} \,h_N \big(1 + \varphi_I (z^\prime,z)\big), \;\mbox{with} \;|\varphi_I| \le c\,\dis\frac{d_N}{h_N}\,.
\end{equation}

\medskip\n
With Khasminski's lemma, cf.~(\ref{2.46}) of \cite{DembSzni06}, and the fact that $\sup_{z^\prime \in \IZ} \,E^\gamma_{z^\prime} [L^z_{T_{\wt{I}}}] = E^\gamma_z [L^z_{T_{\wt{I}}}] \le c \,h_N$, we see that
\begin{equation}\label{2.56}
\sup\limits_{z^\prime \in \IZ} \;E^\gamma_{z^\prime} \Big[\exp \Big\{ \dis\frac{c}{h_N} \;L^z_{T_{\wt{I}}}\Big\}\Big] \le 2 \,.
\end{equation}

\medskip\n
Coming back to (\ref{2.53}) we find that in the notation of (\ref{2.16})
\begin{equation}\label{2.57}
\underset{N}{\overline{\lim}} \;\sup\limits_I \;E^\gamma_0 [M^2_{k_*}] / a^2_N \le \underset{N}{\overline{\lim}} \;c\, \Big(\dis\frac{h_N}{a_N}\Big)^2 \;\sup\limits_I \;E^\gamma_0 \Big[\dsl_{1 \le k \le k_*} 1\{Z_{R_k} \in I\Big\}\Big] \stackrel{(\ref{2.1}),(\ref{2.15})}{=} 0\,.
\end{equation}
Moreover we also find that
\begin{equation}\label{2.58}
\begin{array}{l}
\underset{N}{\overline{\lim}} \;\sup\limits_I\,\Big| \, E^\gamma_0 \;\Big[\dsl_{1 \le k \le k_*} \,1\{Z_{R_k} \in I\} \,\Big(E_{Z_{R_k}}[L^z_{T_{\wt{I}}}] - \dis\frac{h_N}{\gamma}\Big)\Big] / a_N \Big| \stackrel{(\ref{2.55})}{\le}
\\[2ex]
\underset{N}{\overline{\lim}} \;\sup\limits_I \;c\;\dis\frac{d_N}{h_N} \;\dis\frac{h_N}{a_N} \;E_0^\gamma \,\Big[\dsl_{1 \le k \le k_*} \,1\{Z_{R_k} \in I\}\Big] \stackrel{(\ref{2.1}),(\ref{2.15})}{=} \;0\,.
\end{array}
\end{equation}

\medskip\n
The combination of (\ref{2.51}), (\ref{2.57}) and (\ref{2.58}) yields our claim (\ref{2.16}).
\end{proof}

\begin{remark}\label{rem2.3} \rm In the remainder of the article we will mostly be interested in the cases where $a_N$ in (\ref{2.1}) equals $N^d$, and $\gamma$ in (\ref{2.1}) is either $(d+1)^{-1}$, corresponding to the law of the vertical component of the walk on $E$, or $1$, corresponding to simple random walk on $\IZ$, which coincides with the law of $\wh{Z}_\point$ in (\ref{1.5}). \hfill $\square$

\end{remark}

\section{The coupling construction}
\setcounter{equation}{0}

In this section we first specify grids on $\IZ$ which are adapted to the special points $x_1,\dots,x_M$ in $E$ of Theorem \ref{theo0.1} and satisfy the requirements of the last section. The choice of these grids determines a system of excursions for the vertical component of the walk on $E$ with sequences of return and departure times still denoted by $R_k$, $D_k$, $k \ge 1$, (with an abuse of notation). We choose the sequence $h_N$ large enough so that the $\IT$-component of the walk between the successive  $R_k$, $D_k$, has time to homogenize. This is the basis for a coupling construction, where we introduce auxiliary processes $\wt{X}_\point^k$,  $k \ge 2$, which are close to the excursions $X_{(R_k + \cdot) \wedge D_k}$, $k \ge 2$, and conditionally independent given the $Z_{R_k}$, $Z_{D_k}$, $k \ge 1$. Their respective conditional distributions $P_{Z_{R_k}, Z_{D_k}}$, $k \ge 2$, see (\ref{3.19}), have starting points uniformly distributed on $\IT \times \{Z_{R_k}\}$, and are quite handy for the type of calculations we will perform in the next section when proving Theorem \ref{theo0.1}, see in particular Lemma \ref{lem4.2}. The main result of this section appears in Proposition \ref{prop3.3}. The convention concerning constants explained at the end of the Introduction is again in force. In the notation of Section 2 we now choose
\begin{equation}\label{3.1}
a_N = N^d, \;\mbox{and} \;\gamma = (d+1)^{-1}\,.
\end{equation}

\n
In particular  $Q_z^{\gamma = (d+1)^{-1}}$ coincides with the law of the vertical component $Z_\point$ of $X_\point$ under any $P_x$, $x \in \IT \times \{z\}$, and under $P$ if $z = 0$. In the notation of Theorem \ref{theo0.1} and below (\ref{0.10}), we have
\begin{equation}\label{3.2}
\{v_i, 1 \le i \le M\} = \{w_1,\dots , w_L\}, \;\mbox{where} \;w_1 < \dots < w_L\,,
\end{equation}

\n
and for each $\ell \in \{1,\dots , L\}$, we write
\begin{equation}\label{3.3}
\cI_\ell = \{ 1 \le i \le M; \;v_i = w_\ell \}\,.
\end{equation}

\n
As a consequence of (\ref{0.6}), we see that
\begin{equation}\label{3.4}
\lim_N \;\max\limits_{1 \le \ell \le L} \; \max\limits_{i,i^\prime \in \cI_\ell} \;\dis\frac{|z_i - z_{i^\prime}|}{N^d} = 0 \,,
\end{equation}

\medskip\n
and we choose sequences $h_N,d_N$ tending to infinity so that
\begin{equation}\label{3.5}
\begin{array}{ll}
{\rm i)} &h_N = o(N^d), \; h_N \ge N\;(\log N)^2 
\\[2ex]
{\rm ii)} &d_N = o(h_N), \; 1 + \max\limits_{1 \le \ell \le L} \;\max\limits_{i,i^\prime \in \cI_\ell} \;|z_i - z_{i^\prime}| = o(d_N)\,.
\end{array}
\end{equation}

\n
In view of (\ref{3.4}) and the fact that $d \ge 2$, such a choice is possible. We also choose
\begin{equation}\label{3.6}
z^*_\ell(N) = \max \{z_i (N), i \in \cI_\ell\}\,.
\end{equation}

\n
From now on we assume $N$ large enough such that with the present choices,
\begin{equation}\label{3.7}
2 \max\limits_{i \in \cI_\ell} \;|z_i - z^*_\ell |\le d_N, \;\mbox{and (\ref{2.3}) is fulfilled.}
\end{equation}

\noindent
We are thus in the set-up of Section 2, with a grid $\cG$ defined by (\ref{2.4}), and we also write for $1 \le \ell \le L$, 
\begin{align}
I_\ell & = z^*_\ell + [-d_N, d_N]  \subseteq \wt{I}_\ell = z^*_\ell+ (-h_N, h_N)\,,
\label{3.8}
\\[1ex]
B_\ell & = \IT \times I_\ell \subset \wt{B}_\ell = \IT \times \wt{I}_\ell\,, \label{3.9}
\end{align}

\medskip\n
as well as, cf.~(\ref{2.5}),
\begin{equation}\label{3.10}
\cC = \IT \times C, \quad \cO = \IT \times O\,.
\end{equation}

\medskip\n
The successive times of return of the walk $X_\point$ to $\cC$ and departure from $\cO$ naturally coincide with the successive times of return to $C$ and departure from $O$ of the vertical component $Z_\point$ of the walk. With an abuse of notation we still denote them with $R_k$, $D_k$, $k \ge 1$, so that
\begin{equation}\label{3.11}
\begin{array}{l}
R_1 = H_\cC, \;D_1 = T_{\cO} \circ \theta_{R_1} + R_1, \;\mbox{and for} \; k \ge 1
\\[1ex]
R_{k+1} = R_1 \circ \theta_{D_k} + D_k, \;D_{k+1} = D_1 \circ \theta_{D_k} + D_k \,.
\end{array}
\end{equation}

\medskip\n
Given $x \in E \backslash \cO$ we define
\begin{equation}\label{3.12}
\mbox{$z_+(x) > \pi_\IZ(x) > z_-(x)$ the closest points to $\pi_\IZ(x)$ in $\cG + \{-d_N, d_N\}$}\,.
\end{equation}

\n
For $x$ as above the distribution of $X_{R_1}$ under $P_x$ is concentrated on $\IT \times \{z_+(x), z_-(x)\}$, and as we now see it is close to a convex combination of $\nu_{z_+(x)}$ and $\nu_{z_-(x)}$, in the notation of (\ref{1.2}), when $N$ is large.
\begin{lemma}\label{lem3.1}
For large $N$, for all $x \in E \backslash \cO$, $z^\prime \in \{z_+(x), z_-(x)\}$, $x^\prime \in \IT \times \{z^\prime\}$, one has
\begin{equation}\label{3.13}
| P_x [X_{R_1} = x^\prime \,|\,Z_{R_1} = z^\prime ] - N^{-d} \,| \,\le c\,N^{-4d}\,.
\end{equation}
\end{lemma}

\begin{proof}
With (\ref{1.10}) and the fact that the discrete skeleton of $\ov{X}_\point$ is distributed as $X_\point$, we see that $X_{R_1} = (Y_{R_1}, Z_{R_1})$ under $P_x$ has same distribution as $(\ov{Y}_{\hspace{-0.5ex} \ov{R}_1}$, $\ov{Z}_{\ov{R}_1})$ under $P^{\IT}_y \times P^{\IZ}_z$, if $x = (y,z)$ and $\ov{R}_1 = \inf\{t \ge 0, \ov{Z}_t \in C\}$. Writing $x^\prime = (y^\prime,z^\prime)$, the conditional probability in (\ref{3.13}) equals:
\begin{equation}\label{3.14}
P^{\IT}_y \times P^{\IZ}_z [\ov{Y}_{\hspace{-0.5ex} \ov{R}_1} = y^\prime \,| \,\ov{Z}_{\ov{R}_1} = z^\prime] = E^{\IZ}_z \big[ \mu^y_{\ov{R}_1} \;(y^\prime) \, |\,  \ov{Z}_{\ov{R}_1} = z^\prime \big]\,,
\end{equation}

\n
where for $t \ge 0$, we have set $\mu^y_t (\cdot) = P_y^{\IT} [\ov{Y}_t = \cdot ]$.

\medskip
Using standard estimates on  the displacement of simple random walk in continuous time on $\IZ$, see for instance (\ref{2.22}) of \cite{Szni08a}, and the fact that $|z - z^\pm (x)| \ge \frac{1}{2}\;N(\log N)^2$, cf.~(\ref{3.5}) i) and (\ref{2.3}), we have:
\begin{equation}\label{3.15}
P_z^{\IZ} [\ov{R}_1 \le N^2 (\log N)^2] \le c\, N^{-4d} \,,
\end{equation}
and thanks to (\ref{2.6}) i),
\begin{equation}\label{3.16}
P_z^{\IZ} [\ov{Z}_{\ov{R}_1} = z^\prime] \ge c\,.
\end{equation}

\n
Hence the expression in (\ref{3.13}) is smaller than
\begin{equation}\label{3.17}
E^\IZ_z \big[\big|  \mu^y_{\ov{R}_1} \;(y^\prime) - N^{-d} \, \big|1\{\ov{R}_1 > N^2 (\log N)^2\} \big| \,\ov{Z}_{\ov{R}_1} = z^\prime\big] + c\,N^{-4d}\,.
\end{equation}

\n
It follows from Lemma \ref{lem1.1} of \cite{Szni08a}, that for $t \ge t_\IT = \lambda_\IT^{-1} \log (2 |\IT|)$, where $\lambda_\IT$ stands for the spectral gap of the walk $\ov{Y}_\point$ on $\IT$, see (\ref{1.8}) of \cite{Szni08a}, one has
\begin{equation}\label{3.18}
|\mu^y_t(y^\prime) \,N^d -1 | \le \fr \;\exp\{ - (t - t_\IT) \,\lambda_\IT\},  \;\mbox{for} \;t \ge t_\IT \,.
\end{equation}

\n
The $d$ components of $\ov{Y}_\point$ are independent continuous time random walks on $\IZ/N\IZ$ with jump rate equal to $2$, and $\lambda_\IT$ coincides with the spectral gap of simple random walk on $\IZ/N \IZ$ with jump rate $2$ which is bigger than $c\,N^{-2}$ for $N \ge 2$. Hence we have $t_\IT \le c\,N^2 \log (2 N^d)$, and collecting (\ref{3.17}), (\ref{3.18}), we obtain (\ref{3.13}).
\end{proof}

\begin{remark}\label{rem3.2} \rm
Note that the exponent $-4d$ in the right-hand side of (\ref{3.13}) can be replaced by an arbitrarily large negative exponent by adjusting constants. This specific choice will simply be sufficient for our purpose in what follows. \hfill $\square$
\end{remark}

We now come to the main coupling construction of this section. Given $z \in C$ and $z^\prime$ with $P_{\nu_z} [Z_{D_1} = z^\prime] > 0$, (in other words $z^\prime \in \partial \wt{I}$ if $\wt{I}$ is the connected component of $O$ containing $z$),  we introduce the notation:
\begin{equation}\label{3.19}
P_{z,z^\prime} = P_{\nu_z} [ \cdot \,|Z_{D_1} = z^\prime]\,.
\end{equation}

\n
The law of $Z_\point$ under $P_x$ is the same for all $x \in \IT \times \{z\}$, and we also have:
\begin{equation}\label{3.20}
P_{z,z^\prime} = \dis\int d \nu_z(x) \,P_x [ \cdot \,|Z_{D_1} = z^\prime]\,.
\end{equation}

\n
With hopefully obvious notation the principal result of this section is
\begin{proposition}\label{prop3.3}
For large $N$ one can construct on an auxiliary probability space $(\wt{\Omega}, \wt{\cA}, \wt{P})$ a $\IZ$-valued process $Z_\point$ and $\IT$-valued processes $Y_\point$ and $\wt{Y}_\point^k$, $k \ge 2$, such that
\begin{align}
&\mbox{$X_\point = (Y_\point, Z_\point)$ has the same distribution under $\wt{P}$ as the walk on $E$ under $P$,} \label{3.21}
\\[1ex]
&\mbox{under $\wt{P}$ conditionally on $Z_{\cdot \wedge D_1}$, $Z_{R_k}$, $Z_{D_k}$, $k \ge 2$, the processes}\label{3.22}
\\[-0.5ex]
&\mbox{$\wt{X}_\point^k = (\wt{Y}^k_\point, Z_{(R_k + \cdot) \wedge D_k}), k \ge 2$, are independent with the same law as}\nonumber
\\[-0.5ex]
&\mbox{$X_{\cdot \wedge D_1}$ under $P_{Z_{R_k}, Z_{D_k}}$, $k \ge 2$.} \nonumber
\\[2ex]
&\wt{P} [Y_{(R_k + \cdot) \wedge D_k} \not= \wt{Y}_\point^k] \le c \,N^{-3d}, \;\mbox{for $k \ge 2$}\,. \label{3.23}
\end{align}
\end{proposition}

\begin{proof}
Throughout the proof we assume $N$ sufficiently large so that (\ref{3.7}) and (\ref{3.13}) hold. Given $x = (y,z)$ in $\partial \cO$, we write
\begin{equation}\label{3.24}
\kappa_z(dz^\prime) = P_x [Z_{R_1} \in dz^\prime] = Q_z^{\gamma = (d+1)^{-1}}[Z_{R_1} \in dz^\prime]\,.
\end{equation}

\medskip\n
In the notation of (\ref{3.12}), the above distribution is concentrated on $\{z_+(x), z_-(x)\}$. With Lemma \ref{lem3.1}, the total variation distance between the conditional distribution of $X_{R_1}$ under $P_x$, given that $Z_{R_1} = z_\pm(x)$ and $\nu_{z_\pm(x)}$ is smaller than $c\,N^{d-4d} = c\,N^{-3d}$. With Theorem 5.2, p.~19 of Lindvall \cite{Lind92}, we can construct for any $x \in \partial \cO$ a probability
\begin{equation}\label{3.25}
\rho_x(dx^\prime,d\wt{x}) \;\mbox{on}\;\big\{(x^\prime, \wt{x}) \in E^2; \;\pi_\IZ(x^\prime) = \pi_{\IZ}(\wt{x}) \in \{z_+ (x), z_-(x)\}\big\}\,,
\end{equation}
such that under $\rho_x$
\begin{align}
&\mbox{the first component has same distribution as $X_{R_1}$ under $P_x$,}\label{3.26}
\\[1ex]
&\mbox{the second component has the law $\nu \otimes \kappa_z$, (see (\ref{1.2}) for the notation),}\label{3.27}
\end{align}
and moreover
\begin{equation}\label{3.28}
\rho_x (\{x^\prime \not= \wt{x}\}) \le c\,N^{-3d}\,.
\end{equation}

\medskip\n
The auxiliary space we consider is $\wt{\Omega} = \cT_\IT \times \cT_\IZ \times (\cT^f_\IT)^{[2,\infty)}$, where $\cT_\IT, \cT_\IZ$ are the canonical spaces of $\IZ$- and $\IT$-valued trajectories with jumps of $|\cdot |$-size at most 1, and $\cT^f_\IT$ is the (countable) subset of $\cT_\IT$ of trajectories, which are constant after a finite time. We define $\cT^f_\IZ$ analogously. We endow $\wt{\Omega}$ with the canonical product $\sigma$-algebra $\wt{\cA}$, and write $Y_\point, Z_\point$ and $\wt{Y}_\point^k$, $k \ge 2$, for the canonical processes as well as $X_\point = (Y_\point, Z_\point$). The probability $\wt{P}$ on $(\wt{\Omega}, \wt{\cA})$ is constructed as follows.
\begin{align}
&\mbox{The law of $X_{\cdot \wedge D_1}$ under $\wt{P}$ coincides with $P[X_{\cdot \wedge D_1} \in \cdot].$}\label{3.29}
\\[3ex]
&\mbox{The conditional law $\wt{P}[X_{(D_1 + \cdot ) \wedge R_2} \in dw, (\wt{Y}^2_0,Z_{R_2}) \in d\wt{x} \, | X_{\cdot \wedge D_1}]$}\label{3.30}
\\
&\mbox{is equal to $P_{X_{D_1}}[(X_{\cdot \wedge R_1}) \in dw | X_{R_1} = x^\prime] \;\rho_{X_{D_1}}(dx^\prime, d\wt{x}).$}\nonumber
\end{align}

\n
The above two steps specify the law of $X_{\cdot \wedge R_2}$, $\wt{Y}_0^2$ under $\wt{P}$. We then proceed as follows.
\begin{align}
&\mbox{Conditionally on $(X_{\cdot \wedge R_2})$, $ \wt{Y}^2_0$, the law of $X_{(R_2 + \cdot) \wedge D_2}$,} \label{3.31}
\\
&\mbox{under $\wt{P}$ is $P_{X_{R_2}}[(X_{\cdot \wedge T_\cO}) \in dw]$,}\nonumber
\\[3ex]
&\mbox{If $\wt{Y}^2_0 = Y_{R_2} \big(= \pi_\IT (X_{R_2})\big)$, then $\wt{Y}_\point^2 = Y_{(R_2 + \cdot) \wedge D_2}$, $\wt{P}$-a.s.}\,.\label{3.32}
\\[3ex]
&\mbox{If $\wt{Y}^2_0 \not= Y_{R_2}$, then conditionally on $X_{\cdot \wedge D_2}$, $\wt{Y}^2_0$, the law}\label{3.33}
\\
&\mbox{of $\wt{Y}^2_\point$ under $\wt{P}$ is $P_{(\wt{Y}^2_0,Z_{R_2})} [Y_{\cdot \wedge T_\cO} \in dw^\prime | Z_{\cdot \wedge T_\cO} = w(\cdot)]$,} \nonumber
\\
&\mbox{where $w(\cdot) = Z_{(R_2 + \cdot) \wedge D_2} = \pi_\IZ (X_{(R_2 + \cdot) \wedge D_2})$}\,.\nonumber
\end{align}

\medskip\n
The above steps specify the law of $(X_{\cdot \wedge D_2}, \wt{Y}_\point^2)$ under $\wt{P}$. We then proceed using the kernel in the last line of (\ref{3.30}) with $X_{D_2}$ in place of $X_{D_1}$ to specify the conditional law under $\wt{P}$ of $(X_{\cdot \wedge R_3})$, $\wt{Y}^3_0$, given $X_{\cdot \wedge D_2}$, $\wt{Y}_\point^2$, and so on and so forth to obtain the full law $\wt{P}$.

\medskip
With the above construction the claim (\ref{3.21}) follows in a straightforward fashion using (\ref{3.26}). The claim (\ref{3.23}) follows from (\ref{3.28}) and the statements (\ref{3.30}), (\ref{3.32}) and their iterations for arbitrary $k \ge 2$. To prove (\ref{3.22}) it suffices to show by induction that for bounded functions $H$ on $\cT^f_\IZ$, $g_2,\dots,g_k$ on $\cT^f_\IT \times \cT^f_\IZ$ and $h_2,\dots,h_k$ on $\IZ^2$, writing $\wt{E}$ for the $\wt{P}$-expectation, we have the identity
\begin{equation}\label{3.34}
\begin{array}{l}
\wt{E}\big[H(Z_{\cdot \wedge D_1}) \,g_2(\wt{X}^2_\point) \,h_2(Z_{R_2},Z_{D_2}) \dots g_k(\wt{X}^k_\point) \,h_k (Z_{R_k},Z_{D_k})] =
\\[1ex]
\wt{E}\big[H(Z_{\cdot \wedge D_1})\,E_{Z_{R_2},Z_{D_2}}[g_2(X_{\cdot \wedge D_1})] \,h_2(Z_{R_2},Z_{D_2}) \dots 
\\[1ex]
E_{Z_{R_k},Z_{D_k}} [g_k (X_{\cdot \wedge D_1})] \,h_k (Z_{R_k},Z_{D_k})\big]\,.
\end{array}
\end{equation}

\medskip\n
The above equality is obvious when $k = 1$. Assume it holds up to $k \ge 1$, and note that
\begin{equation}\label{3.35}
\begin{array}{l}
\wt{E}\big[H(Z_{\cdot \wedge D_1}) \,g_2(\wt{X}^2_\point) \,h_2(Z_{R_2},Z_{D_2}) \dots g_{k+1}(\wt{X}^{k+1}_\point) \,h_{k+1}(Z_{R_{k+1}},Z_{D_{k+1}})] =
\\[1ex]
\wt{E}\big[H(Z_{\cdot \wedge D_1}) \,g_2(\wt{x}^2_\point) \,h_2(Z_{R_2},Z_{D_2}) \dots \dis\int  \rho_{X_{D_k}} (dx^\prime, d \wt{x}) 
\\[1ex]
E_{\wt{x}} [g_{k+1}(X_{\cdot \wedge D_1})\,h_{k+1} (Z_0,Z_{D_1})]\big]\,,
\end{array}
\end{equation}

\medskip\n
where we used that in the first line the terms to the left of $g_{k+1}$ depend on $X_{\cdot \wedge D_k}$, $\wt{Y}^2_\point, \dots, \wt{Y}^k_\point$, together with the statements (\ref{3.30}) - (\ref{3.33}). Observe that
\begin{equation}\label{3.36}
\begin{array}{l}
\dis\int \rho_{X_{D_k}}(dx^\prime, d \wt{x})  \,E_{\wt{x}}[g_{k+1} (X_{\cdot \wedge D_1})\,h_{k+1} (Z_0,Z_{D_1})] \stackrel{(\ref{3.27})}{=}
\\[1ex]
\dis\int \kappa_{Z_{D_k}}(dz^\prime) \,E_{\nu_{z^\prime}}[g_{k+1} (X_{\cdot \wedge D_1})\,h_{k+1} (Z_0,Z_{D_1})] = \Psi_k(Z_{D_k})\,,
\end{array}
\end{equation}

\medskip\n
where for $z \in \partial O$, in the notation of (\ref{3.19}) and of Section 2, we have set:
\begin{equation}\label{3.37}
\Psi_k(z) = \dis\int \kappa_z (dz^\prime) \,E_{z^\prime}^{\gamma = (d+1)^{-1}}\big[h_{k+1}(z^\prime,Z_{D_1})\,E_{z^\prime,Z_{D_1}}[g_{k+1}(X_{\cdot \wedge D_1})]\big]\,.
\end{equation}

\n
Inserting the identity (\ref{3.36}) in the second  line of (\ref{3.35}) we can use the induction hypothesis to find that the first line of (\ref{3.35}) equals
\begin{equation*}
\begin{array}{l}
\wt{E}[H(Z_{\cdot \wedge D_1}) \,E_{Z_{R_2},Z_{D_2}} [g_2(X_{\cdot \wedge D_1})]\,h_2(Z_{R_2},Z_{D_2}) \dots h_k (Z_{R_k}, Z_{D_k}) \,\Psi_k(Z_{D_k})] =
\\[1ex]
\wt{E}[H(Z_{\cdot \wedge D_1}) \,E_{Z_{R_2},Z_{D_2}} [g_2(X_{\cdot \wedge D_1})]\,h_2(Z_{R_2},Z_{D_2}) \dots h_k (Z_{R_k}, Z_{D_k}) \,E_{Z_{R_{k+1}},Z_{D_{k+1}}} [g_{k+1}(X_{\cdot \wedge D_1})]
\\[1ex]
\quad \; h_{k+1} (Z_{R_{k+1}},Z_{D_{k+1}})]\,,
\end{array}
\end{equation*}

\n
where we used the fact that, see (\ref{3.21}), the distribution of $Z_\point$ under $\wt{P}$ coincides with $Q_0^{\gamma = (d+1)^{-1}}$, and the strong Markov property at times $R_{k+1}$ and $D_k$. This concludes the proof by induction of (\ref{3.34}), and yields (\ref{3.22}).
\end{proof}

\begin{remark}\label{rem3.4} \rm
With a slight variation on the proof of (\ref{3.22}) one can show that under $\wt{P}$, conditionally on $Z_{\cdot \wedge D_1}, Z_{(R_k + \cdot)\wedge D_k}$, $k \ge 2$, the $\wt{Y}^k_\point$, $k \ge 2$, are independent with respective distribution
\begin{equation*}
P_{\nu_{Z_{R_k}}} [Y_{\point \wedge D_1} \in dw^\prime \,|\, Z_{\cdot \wedge D_1} = w_k(\cdot)], \;\mbox{where}\;w_k(\cdot) = Z_{(R_k + \cdot ) \wedge D_k} \,.
\end{equation*}

\medskip\n
One can then quickly recover (\ref{3.22}), but we will not need this result in what follows. 

~ \hfill $\square$
\end{remark}

\section{Denouement}
\setcounter{equation}{0}

We will now provide the proof of Theorem \ref{theo0.1} in this section. This amounts to showing the convergence of expectations of the kind which appear in (\ref{4.1}) below. With the results of Section 2 and Lemma \ref{lem4.1}, we are able to replace in (\ref{4.1}) the times $T_N$ with the stopping times $D_{k_*}$, when $\rho$ in (\ref{2.11}) is chosen equal to $\alpha$ in (\ref{0.7}), and the local times $L^{z_i}_{T_N}$ with sums $(d+1) \,h_N \sum_{1 \le k \le k_*} 1\{Z_{R_k} \in I_\ell\}$, where $z_i \in I_\ell$. With the results of Section 3, we can replace the excursions $X_{(R_k + \cdot) \wedge D_k}$ with auxiliary excursions $\wt{X}_\point^k$, which are very well behaved under the conditional law where all variables $Z_{R_k},Z_{D_k}$, $k \ge 1$, are given. The main task is then to investigate how the excursions for which $Z_{R_k} \in \bigcup_{1 \le \ell \le L} I_\ell$, affect the local picture of the trace left by the walk in the neighborhood of the points $x_i, 1 \le i \le M$. This is done in the key Lemma \ref{lem4.2}, which shows that the anisotropic grids we consider and their corresponding system of excursions, are well suited for this task.

\bigskip\n{\it Proof of Theorem {\rm \ref{theo0.1}}:} In the notations of Theorem \ref{theo0.1}, our claim will follow once we show that for $K_1,\dots,K_M$ finite subsets of $\IZ^{d+1}$ and $\lambda_1,\dots,\lambda_M \ge 0$, one has
\begin{equation}\label{4.1}
\begin{array}{l}
\lim\limits_N A_N = A, \;\mbox{where}
\\[1ex]
A_N = E \Big[\mbox{\small $\prod\limits^M_{i=1}$} \,1 \{H_{x_i + K_i} > T_N\} \;\exp\Big\{ - \dsl^M_{i=1} \;\dis\frac{\lambda_i}{N^d} \;L^{z_i}_{T_N}\Big\}\Big], \;\mbox{for $N \ge 1$, and}
\\[1ex]
A = E^W\Big[\exp\Big\{- \dsl^M_{i=1} \,(d+1) \,L\Big(v_i, \mbox{\f $\dis\frac{\alpha}{d+1}$}\Big) \;({\rm cap} (K_i) + \lambda_i)\Big\}\Big]\,.
\end{array}
\end{equation}

\medskip\n
Indeed one straightforwardly uses the usual arguments relating the convergence of Laplace functionals to weak convergence, cf.~\cite{Chun74}, p.~189-191, the compactness of $\{0,1\}^{\IZ^{d+1}}$, as well as the fact that on $\{0,1\}^{\IZ^{d+1}}$, the collection of events $\{\o(x) = 1$, for all $x \in K\}$, as $K$ varies over finite subsets of $\IZ^{d+1}$, is a $\pi$-system generating the canonical product $\sigma$-algebra. Further using the boundedness and monotonicity in $T_N$ of the expression inside the expectation defining $A_N$, the continuity in $\alpha$ of the expectation defining $A$, and (\ref{0.7}), the claim (\ref{4.1}) will follow once we show it for special sequences of the form
\begin{equation}\label{4.2}
T_N = [\alpha \,N^{2d}], \;\mbox{with $\alpha > 0$}\,.
\end{equation}

\medskip\n
From now on we assume that $T_N$ is of the above form and will write $T$ in place of $T_N$ for simplicity. This agrees with the notation (\ref{2.12}), when picking $\rho = \alpha$ in (\ref{2.11}) and $a_N = N^d$, as we will implicitly do in most of what follows. Unless stated otherwise we tacitly choose $\gamma = (d+1)^{-1}$ in (\ref{2.7}). We choose the sequences $h_N, d_N$, cf.~(\ref{3.5}), as well as the inhomogeneous grids and the corresponding stopping times $R_k, D_k, k \ge 1$, as explained at the beginning of Section 3. These choices are implicit when referring to notation as in (\ref{2.10}), (\ref{2.12}) or numbered constants as in (\ref{2.20}) or (\ref{2.43}). Finally we assume $N$ large enough so that (\ref{3.7}) holds, Proposition \ref{prop3.3} applies, and cf.~(\ref{3.5}), (\ref{3.7}):
\begin{equation}\label{4.3}
\mbox{for $1 \le \ell \le L$, and $i \in \cI_\ell$, $x_i + K_i \subseteq \IT \times I_\ell$}\,.
\end{equation}

\begin{lemma}\label{lem4.1}
\begin{align}
&\lim\limits_N \;\sup\limits_{z \in \IZ, x \in E} \;P_{\nu_z} [H_x \le D_{k^* - k_*}] = 0, \;\mbox{(see {\rm  (\ref{1.2})} for the notation)} \,. \label{4.4}
\\[1ex]
&\lim\limits_N \;P\big[H_{\bigcup\limits_{1 \le i \le M} (x_i + K_i)} \le D_1 \big] = 0\,. \label{4.5}
\\[1ex]
&\lim\limits_N \;E\big[\big| \mbox{\small $\prod\limits_{1 \le i \le M}$}1 \{H_{x_i + K_i} > D_{k^*}\big\} -  \mbox{\small $\prod\limits_{1 \le i \le M} $}1 \{H_{x_i + K_i} > D_{k_*}\big\}\big|\big] = 0\,.\label{4.6}
\end{align}
\end{lemma}

\medskip
\begin{proof}
We begin with the proof of (\ref{4.4}). As in (\ref{2.43}) we see that
\begin{equation}\label{4.7}
\begin{array}{l}
\underset{N}{\overline{\lim}} \;\sup\limits_z \;\Big\{P_{\nu_z} \Big[D_{k^* - k_*} \ge c_3\,t_N \Big[\alpha \;\dis\frac{N^{2d}}{t_N}\Big]^{3/4}\Big]\Big\} \le 
\\[2ex]
\underset{N}{\overline{\lim}} \; \exp\Big\{- c_0\,c_3 \Big[\alpha\;\dis\frac{N^{2d}}{t_N}\Big]^{3/4} + 2 (\log 2) \;\Big[\alpha \;\dis\frac{N^{2d}}{t_N}\Big]^{3/4}\Big\} = 0\,.
\end{array}
\end{equation}

\medskip\n
Writing $\gamma_N = [c_3\,t_N [\alpha \;\frac{N^{2d}}{t_N}]^{3/4}]$, so that with the notation below (\ref{4.2}), $\gamma_N = o(T)$, as $N$ tends to infinity, we see that
\begin{equation}\label{4.8}
\begin{array}{l}
P_{\nu_z} [H_x \le \gamma_N] \stackrel{(\ref{1.9})}{=} P_{\nu_z} [\ov{H}_x \le \ov{\sigma}_{\gamma_N}]  \le c\;E_{\nu_z} \Big[\dis\int_0^{\ov{\sigma}_{\gamma_N + 1}} 1\{\ov{X}_t = x\}\,dt\Big] \le
\\
\\[-1ex]
c\,E_{\nu_z} [\ov{\sigma}_{\gamma_N + 1} \ge \gamma_N +1, \,\ov{\sigma}_{\gamma_N + 1}]  + c\,E_{\nu_z} \Big[\dis\int_0^{\gamma_N + 1} 1\{\ov{X}_t = x\}\,dt\Big] \,.
\end{array}
\end{equation}

\n
The variable $\ov{\sigma}_n$ is a sum of $n$ independent exponential variables with parameter $2d + 2$. The first term in the last member of (\ref{4.8}) does not depend on $z$, and for $0 < \lambda < d+1$, is smaller than:
\begin{equation*}
c\,\lambda^{-1}\,E[e^{2 \lambda \,\ov{\sigma}_{\gamma_N + 1}}] \,e^{-\lambda(\gamma_N + 1)} \le c \,\lambda^{-1} \exp\Big\{(\gamma_N + 1) \Big[- \lambda + \log \Big(\dis\frac{1}{1 - \frac{\lambda}{d+1}}\Big)\Big] \Big\}\,.
\end{equation*}

\medskip\n
Choosing $\lambda$ close to zero, we see that the last expression tends to zero as $N$ goes to infinity. As for the last term of (\ref{4.8}), writing $x = (y,z^\prime)$, and using (\ref{1.10}) together with the fact that $\nu$ is the stationary distribution of $\ov{Y}_\point$, we see it equals:
\begin{equation}\label{4.9}
\begin{array}{l}
\dis\frac{c}{N^d} \;E^\IZ_z \Big[ \dis\int_0^{\gamma_N + 1} 1\{\ov{Z}_t = z^\prime\} \,dt\Big] \stackrel{\rm strong \; Markov}{\le} \dis\frac{c}{N^d} \;E^\IZ_{z^\prime} \Big[\dis\int_0^{\gamma_N + 1} 1\{\ov{Z}_t = z^\prime\}\,dt\Big] \le
\\[2ex]
\dis\frac{c}{N^d} \;\dis\int_0^{\gamma_N + 1} \;\dis\frac{dt}{\sqrt{t}} \le c\;\dis\frac{\sqrt{\gamma_N + 1}}{N^d} \; \underset{N \r \infty}{\longrightarrow} \;0\,,
\end{array}
\end{equation}
where we used the bound
\begin{equation}\label{4.10}
\sup\limits_{z,z^\prime \in \IZ} \;P_{z^\prime}^{\IZ} [\ov{Z}_t = z] \le \dis\frac{c}{\sqrt{t}}\,, \;\mbox{for $t > 0$}\,,
\end{equation}

\medskip\n
which is a straightforward consequence of (\ref{2.42}), when $\gamma =1$, and an exponential bound on the probability that a Poisson variable of intensity $2t$ does not belong to the interval $[t,4t]$. We have thus shown that
\begin{equation}\label{4.11}
\lim\limits_N \;\sup\limits_{z \in \IZ, x \in E} \;P_{\nu_z} [H_x \le \gamma_N] = 0\,.
\end{equation}

\medskip\n
Together with (\ref{4.7}), this concludes the proof of (\ref{4.4}). Note that (\ref{4.5}) is a direct consequence of (\ref{4.4}) with the choice $z=0$, since the probability in (\ref{4.5}) is smaller than 
\begin{equation*}
\Big(\dsl^M_1 \;|K_i| \Big) \;\sup\limits_{x \in E} \;P[H_x \le D_1]\,.
\end{equation*}

\medskip\n
Finally the expression in (\ref{4.6}) is smaller than:
\begin{equation}\label{4.12}
\begin{array}{l}
E\Big[1 \Big\{\bigcup\limits_{1 \le i \le M} \Big\{H_{x_i + K_i} \le D_{k^* - k_*}\Big\}\Big\} \circ \theta_{D_{k_*}}\Big] = 
\\[1ex]
E\Big[P_{X_{D_{k_*}}} \Big[\bigcup\limits_{1 \le i \le M} \Big\{H_{x_i + K_i} \le D_{k^* - k_*}\Big\}\Big]\Big]\,.
\end{array}
\end{equation}

\medskip\n
We can apply Lemma \ref{lem1.1}, with the choice $\tau = D_{k_*}$, and find that under $P$, $X_{D_{k_*}}$ has a distribution of the form $\nu(dy) \otimes \gamma(dz)$, where $\gamma$ is a probability on $\IZ$. It thus follows that the right-hand side of (\ref{4.12}) is smaller than
\begin{equation*}
\Big(\dsl_{1 \le i \le M} \;|K_i|\Big) \;\sup\limits_{z \in \IZ, x \in E} P_{\nu_z} [H_x \le D_{k^*- k_*}] \;\underset{N \r \infty}{\stackrel{(\ref{4.4})}{\longrightarrow}} 0\,.
\end{equation*}
This proves (\ref{4.6}).
\end{proof}

\medskip
We now return to the expression for $A_N$ in (\ref{4.1}). In view of (\ref{4.3}) and our tacit assumptions on $N$ stated above (\ref{4.3}), we find that $P$-a.s.,
\begin{align}
&X_{[D_k,R_{k+1} -1 ]} \cap \big(\,\mbox{\small $\bigcup\limits^M_1$}  \,(x_i + K_i)\big) = \emptyset, \;\mbox{for $k \ge 1$, and} \label{4.13}
\\ 
\mbox{on}\; & \big\{X_{R_k} \notin \mbox{\small $\bigcup\limits_{1 \le \ell \le L}$} B_\ell \big\}, \;X_{[R_k,D_k]} \cap \big(\,\mbox{\small $\bigcup\limits^M_1$} \,(x_i + K_i)\big) = \emptyset, \; \mbox{for $k \ge 1$,}\label{4.14}
\end{align}

\medskip\n
In other words the entrance of the walk $X$ in one of the $x_i + K_i$, $1 \le i \le M$, $P$-almost surely can only occur during one of the time intervals $[R_k,D_k]$, $k \ge 1$, with $X_{R_k}$ in $\bigcup_{1 \le \ell \le L} B_\ell$. With (\ref{2.13}) and (\ref{4.5}), (\ref{4.6}), we see that when we replace in $A_N$, cf.~(\ref{4.1}), $\mbox{\f $\prod^M_{i=1}$} 1\{H_{x_i + K_i} > T\}$ with the indicator function of the event
\begin{equation}\label{4.15}
\begin{split}
\big\{ &\mbox{for all $1 \le i \le M$, and $2 \le k \le k_*$, with} 
\\
&X_{R_k} \in \mbox{\small $\bigcup\limits_{1 \le \ell \le L}$} B_\ell, \;X_{[R_k,D_k]} \cap (x_i + K_i) = \phi \big\}\,,
\end{split}
\end{equation}

\n
the difference of the two corresponding expectations tends to zero with $N$. With (\ref{2.14}) and (\ref{2.16}), we see that in the terminology introduced at the beginning of Section 1, $A_N$ is limit equivalent to the expression where we further replace the term inside the exponential by $(d+1) \,\frac{h_N}{N^d} \;\sum^L_{\ell = 1} \;(\sum_{i \in \cI_\ell} \lambda_i) \sum_{1 \le k \le k_*} 1\{Z_{R_k} \in I_\ell\}$, that is:
\begin{equation}\label{4.16}
\begin{split}
E\Big[&\mbox{for} \;1 \le i \le M, 2 \le k \le k_*, \;\mbox{with} \;X_{R_k} \in \mbox{\small $\bigcup\limits_{1 \le \ell \le L}$} B_\ell, X_{[R_k,D_k]} \cap (x_i + K_i) = \phi\, ,
\\
&\exp\Big\{ - \dsl^L_{\ell = 1} (d+1) \;\dis\frac{h_N}{N^d} \;\Big(\dsl_{i \in \cI_\ell} \lambda_i\Big) \;\dsl_{2 \le k \le k_*} 1\{Z_{R_k} \in I_\ell\Big\}\Big\}\Big] \,.
\end{split}
\end{equation}

\medskip\n
Note that $\lim_N k_* \,N^{-3d} = 0$, so that with Proposition \ref{prop3.3}, $A_N$ is limit equivalent to:
\begin{equation}\label{4.17}
\begin{split}
\wt{E}\Big[&\mbox{for} \;1 \le i \le M, \, 2 \le k \le k_*, \;\mbox{with}\;Z_{R_k} \in\mbox{\small $\bigcup\limits_{1 \le \ell \le L}$} \;  I_\ell, \,\wt{X}^k_{[0,D_1]} \cap (x_i + K_i) = \phi \,,
\\
&\exp\Big\{ - \dsl^L_{\ell = 1} (d+1) \;\dis\frac{h_N}{N^d} \;\Big(\dsl_{i \in \cI_\ell} \lambda_i\Big)  \,\dsl_{2 \le k \le k_*} 1\{Z_{R_k} \in I_\ell\}\Big\}\Big] \stackrel{(\ref{3.22})}{=}
\\[1ex]
&\wt{E} \Big[\,\mbox{\small $\prod\limits^L_{\ell=1}$} \;\mbox{\small $\prod\limits_{2 \le k \le k_*}$} \big(1\{Z_{R_k} \notin I_\ell\} + 1\{Z_{R_k} \in I_\ell\}   \,P_{Z_{R_k},Z_{D_k}} [H_{C_\ell} > T_{\wt{B}_\ell}]\big)
\\[1ex]
&\exp\Big\{ - \dsl^L_{\ell = 1} (d+1) \,\dis\frac{h_N}{N^d} \;\Big(\dsl_{i \in \cI_\ell} \lambda_i\Big) \,\dsl_{1 \le k \le k_*} 1\{Z_{R_k} \in I_\ell\} \Big\}\Big]\,,
\end{split}
\end{equation}

\n
where we have set
\begin{equation}\label{4.18}
C_\ell = \mbox{\small $\bigcup\limits_{i \in \cI_\ell}$} \,(x_i + K_i), \;\mbox{for $1 \le i \le L$}\,.
\end{equation}

\n
As mentioned at the beginning of this section one of the key ingredients in the proof of Theorem \ref{theo0.1}, which also motivates the consideration of the kind of inhomogeneous grids we have introduced is the following
\begin{lemma}\label{lem4.2}
For large $N$, $1 \le \ell \le L$, $z_1 \in \partial_{\rm int}  I_\ell, z_2 \in \partial \wt{I}_\ell$,
\begin{equation}\label{4.19}
P_{z_1,z_2} [H_{C_\ell} < T_{\wt{B}_\ell}] = (d+1) \;\dis\frac{h_N}{N^d} \;{\rm cap}_{\wt{B}_\ell} (C_\ell) \big(1 + \Psi_\ell (z_1,z_2)\big)\,,
\end{equation}

\n
with the notation below {\rm (\ref{1.8})} and where
\begin{equation}\label{4.20}
| \Psi_\ell(z_1,z_2)| \le c\;\dis\frac{d_N}{h_N} \;.
\end{equation}
Moreover for $1 \le \ell \le L$, one has
\begin{equation}\label{4.21}
\lim\limits_N \;{\rm cap}_{\wt{B}_\ell} (C_\ell) = \dsl_{i \in \cI_\ell} \;{\rm cap}(K_i)\,.
\end{equation}
\end{lemma}

\begin{proof}
We begin with the proof of (\ref{4.19}), (\ref{4.20}). With $\ell, z_1,z_2$ as above we see that
\begin{equation}\label{4.22}
\begin{array}{l}
P_{z_1,z_2} [H_{C_\ell} < T_{\wt{B}_\ell}]  \stackrel{(\ref{3.19})}{=} P_{\nu_{z_1}} [H_{C_\ell} < T_{\wt{B}_\ell}, Z_{T_{\wt{B}_\ell}} = z_2] / Q_{z_1}^{\gamma = (d+1)^{-1}} [Z_{T_{\wt{I}_\ell}} = z_2] =
\\[1ex]
E_{\nu_{z_1}}\big[H_{C_\ell} < T_{\wt{B}_\ell}, \,Q_{Z_{H_{C_\ell}}}^{\gamma = (d+1)^{-1}} [{Z_{T_{\wt{I}_\ell}}} = z_2]\big] / Q_{z_1}^{\gamma = (d+1)^{-1}} [Z_{T_{\wt{I}_\ell}} = z_2]\,.
\end{array}
\end{equation}
Note that with (\ref{3.5}), (\ref{3.7}),
\begin{equation}\label{4.23}
\sup\limits_{z \in I_\ell} |Q_z^{\gamma = (d+1)^{-1}} [Z_{T_{\wt{I}_\ell}} = z_2] - \fr \;| \le c\;\dis\frac{d_N}{h_N}\;,
\end{equation}

\medskip\n
and hence for large $N$ and any $\ell, z_1,z_2$ as above, with (\ref{4.3})
\begin{equation}\label{4.24}
P_{z_1,z_2} [H_{C_\ell} < T_{\wt{B}_\ell}] = P_{\nu_{z_1}} [H_{C_\ell} < T_{\wt{B}_\ell}] \big( 1 + \Psi^\prime_\ell(z_1,z_2)\big), \;\mbox{where}\;|\Psi^\prime_\ell(z_1,z_2)| \le c\;\dis\frac{d_N}{h_N}\;.
\end{equation}

\medskip\n
On the other hand with the notation for the walk on $E$ introduced below (\ref{1.8}), see also (\ref{1.6}), (\ref{1.8}), we can write
\begin{equation}\label{4.25}
P_{\nu_{z_1}} [H_{C_\ell} < T_{\wt{B}_\ell}] = \dsl_{x,x^\prime} \,\nu_{z_1}(x) \;g_{\wt{B}_\ell} (x,x^\prime) \,e_{C_\ell,\wt{B}_\ell}(x^\prime)\,.
\end{equation}

\n
Using the symmetry of $g_{\wt{B}_\ell}(\cdot,\cdot)$ we find that
\begin{equation}\label{4.26}
g_{\wt{B}_\ell} (x,x^\prime) = g_{\wt{B}_\ell} (x^\prime,x) = 2(d+1) \;E_{x^\prime}\Big[\dis\int_0^{\ov{T}_{\wt{B}_\ell}} 1 \{\ov{X}_t = x\}\,dt\Big]\,.
\end{equation}

\n
Hence for $x^\prime = (y^\prime, z^\prime) \in C_\ell$, we find that
\begin{equation}\label{4.27}
\begin{array}{l}
\dsl_x \nu_{z_1}(x) \,g_{\wt{B}_\ell}(x,x^\prime) \stackrel{(\ref{1.2})}{=} \dis\frac{2(d+1)}{N^d} \;\dsl_{y \in \IT} \;E_{x^\prime} \Big[\dis\int_0^{\ov{T}_{\wt{B}_\ell}} 1\{\ov{X}_t = (y,z_1)\}\,dt\Big] \stackrel{(\ref{1.10})}{=}
\\[1ex]
\dis\frac{2(d+1)}{N^d} \;E_{z^\prime}^{\IZ} \Big[\dis\int_0^{\ov{T}_{\wt{I}_\ell}} 1\{\ov{Z}_t = z_1\}\,dt\Big] = \dis\frac{(d+1)}{N^d} \;\dis\frac{Q^{\gamma = 1}_{z^\prime}[H_{z_1} < T_{\wt{I}_\ell}]}{Q^{\gamma = 1}_{z_1}[\wt{H}_{z_1} > T_{\wt{I}_\ell}]} \;,
\end{array}
\end{equation}

\medskip\n
going back to the discrete time process and counting the expected number of visits to $z_1$ prior to departure from $\wt{I}_\ell$ in the last step.

\medskip
With very similar calculations as in (\ref{2.54}), we see that the last expression for large $N$ equals $(d+1) \,\frac{h_N}{N^d} \, \big(1 + \Psi^{\prime\prime}_\ell(z^\prime,z_1)\big)$ where for $z^\prime \in I_\ell, z_1 \in \partial_{\rm int} I_\ell$, $|\Psi^{\prime\prime}_\ell (z^\prime,z_1)| \le c\,\frac{d_N}{h_N}$. Coming back to (\ref{4.25}), (\ref{4.27}), we see with (\ref{4.3}) that for large $N$, $1 \le \ell \le L$, $z_1 \in \partial_{\rm int} I_\ell$, 
\begin{equation}\label{4.28}
P_{\nu_{z_1}}[H_{C_\ell} < T_{\wt{B}_\ell}]  = (d+1) \;\dis\frac{h_N}{N^d} \;{\rm cap}_{\wt{B}_\ell} (C_\ell) \big(1 + \Psi^{\prime\prime\prime}_\ell (z_1)\big), \;\mbox{where} \;| \Psi^{\prime\prime\prime}_\ell (z_1)| \le c\, \dis\frac{d_N}{h_N}\,.
\end{equation}

\n
Together with (\ref{4.24}), the claim (\ref{4.19}), (\ref{4.20}) follows in a straightforward fashion. 

\medskip
We now turn to the proof of (\ref{4.21}). We define for $1 \le \ell \le L$
\begin{equation}\label{4.29}
\wt{C}_\ell = \mbox{\small $\bigcup\limits_{i \in \cI_\ell}$} B\Big (x_i, \mbox{\f $\dis\frac{N}{10M}$}\Big)\,, 
\end{equation}

\medskip\n 
and assume from now on that $N$ is large enough so that with (\ref{3.5}) i) and (\ref{4.3}), $C_\ell \subseteq \wt{C}_\ell \subseteq \wt{B}_\ell$. From the formulas corresponding to (\ref{1.6}), (\ref{1.7}), when $E$ replaces $\IZ^{d+1}$, we find that
\begin{equation}\label{4.30}
{\rm cap}_{\wt{B}_\ell} (C_\ell) \le {\rm cap}_{\wt{C}_\ell} (C_\ell)   \le \dsl_{i \in \cI_\ell} |K_i|, \;\mbox{for}\;1 \le \ell \le L\,. 
\end{equation}
We first show that
\begin{equation}\label{4.31}
\lim\limits_N \; {\rm cap}_{\wt{C}_\ell} (C_\ell) - {\rm cap}_{\wt{B}_\ell} (C_\ell)  = 0 \,. 
\end{equation}

\medskip\n
Indeed with (\ref{1.6}), (\ref{1.7}) we find that
\begin{equation}\label{4.32}
\begin{array}{l}
0 \le {\rm cap}_{\wt{C}_\ell} (C_\ell) - {\rm cap}_{\wt{B}_\ell} (C_\ell) = \dsl_{x^\prime \in C_\ell} \,P_{x^\prime} [T_{\wt{B}_\ell} > \wt{H}_{C_\ell} > T_{\wt{C}_\ell}] \stackrel{\rm strong \; Markov}{=}
\\[2ex]
\dsl_{x^\prime \in C_\ell} \;E_{x^\prime} \big[\wt{H}_{C_\ell} > T_{\wt{C}_\ell}, P_{X_{T_{\wt{C}_\ell}}} [H_{C_\ell} < T_{\wt{B}_\ell}]\big] \le
\\[2ex]
{\rm cap}_{\wt{C}_\ell}(C_\ell)  \;\sup\limits_{x \in \partial \wt{C}_\ell} \;P_x[H_{C_\ell} < T_{\wt{B}_\ell}] \le \Big(\dsl_{i \in \cI_\ell} |K_i|\Big) \sup\limits_{x \in \partial \wt{C}_\ell} \;P_x [H_{C_\ell} < T_{\wt{B}_\ell}]\,.
\end{array}
\end{equation}

\n
When $N$ is large we also have
\begin{equation}\label{4.33}
  \sup\limits_{x \in \partial \wt{C}_\ell} \;P_x [H_{C_\ell} < T_{\wt{B}_\ell}] \le \Big(\dsl_{i \in \cI_\ell} |K_i|\Big) \sup\limits_{x,x^\prime \,{\rm in}\,  \wt{B}_\ell \atop |x-x^\prime|_\infty \ge \frac{N}{20M}} P_x[H_{x^\prime} < T_{\wt{B}_\ell}]\,.
\end{equation}
The claim (\ref{4.31}) will now follow once we show that
\begin{equation}\label{4.34}
\lim\limits_N \;\wt{\sup} \;P_x[H_{x^\prime} < T_{\wt{B}_\ell}] = 0, \;\mbox{for $1 \le \ell \le L$}\,,
\end{equation}

\medskip\n
where $\wt{\sup}$ stand for the supremum that appears in the right-hand side of (\ref{4.33}). We consider $x = (y,z)$, $x^\prime = (y^\prime,z^\prime)$ in $\wt{B}_\ell$ with $|x - x^\prime |_\infty \ge \frac{N}{20M}$, and $\varepsilon$ in (0,1). Bringing the continuous time walk into play we find:
\begin{equation}\label{4.35}
P_x[H_{x^\prime} < T_{\wt{B}_\ell}] \le P_x[\ov{H}_{x^\prime} \le \ov{T}_{\wt{B}_\ell} \wedge (\varepsilon N^2)] + P_x [\varepsilon N^2 < \ov{H}_x < \ov{T}_{\wt{B}_\ell}]\,.
\end{equation}

\medskip\n
From translation invariance, identifying $B(0, \frac{N}{20M})$ with a  subset of $\IZ^{d+1}$, the first term in the right-hand side of (\ref{4.35}) is smaller than
\begin{equation}\label{4.36}
P_0^{\IZ^{d+1}} \Big[\sup\limits_{0 \le t \le \varepsilon N^2} |\ov{X}_t |_\infty \ge \dis\frac{N}{20M}\Big]\,.
\end{equation}

\medskip\n
The second term after the application of the Markov property at time $\varepsilon N^2$ equals
\begin{equation}\label{4.37}
\begin{array}{l}
E_x \big[\varepsilon N^2 < \ov{H}_{x^\prime} \wedge \ov{T}_{\wt{B}_\ell}, \,P_{\ov{X}_{\ve N^2}} [ \ov{H}_{x^\prime} < \ov{T}_{\wt{B}_\ell}] \big] \le
\\[1ex]
E_x \Big[\varepsilon N^2 < \ov{H}_{x^\prime} \wedge \ov{T}_{\wt{B}_\ell}, \,2(d+1) E_{\ov{X}_{\ve N^2}} \Big[\dis\int_0^{\ov{T}_{\wt{B}_\ell}} 1\{\ov{X}_t = x^\prime\}\,dt\Big]\Big] \stackrel{\rm Markov\; property}{=}
\\[1ex]
2(d+1) \,E_x \Big[\ve N^2 < \ov{H}_{x^\prime} \wedge  \ov{T}_{\wt{B}_\ell}, \dis\int^{\ov{T}_{\wt{B}_\ell}}_{\ve N^2} 1\{\ov{X}_t = x^\prime\}\,dt\Big] \stackrel{(\ref{1.10})}{\le}
\\[1ex]
2(d+1) \,E_y^\IT \;E^\IZ_z \,\Big[\dis\int_{\ve N^2}^\infty 1\{\ov{Y}_t = y^\prime\} \,1\{\ov{Z}_t = z^\prime, t < \ov{T}_{\wt{I}_\ell}\}\,dt\Big]\,.
\end{array}
\end{equation}

\medskip\n
Using classical upper bounds on the heat kernel of simple random walk on $\IZ^d$, see for instance (\ref{2.4}) of \cite{GrigTelc01}, one finds that $\sup_{y,y^\prime \in \IT} P_y^\IT [Y_n = y^\prime] \le c(\ve)\,N^{-d}$, for $n \ge \ve \,N^2$. With exponential bounds on the Poisson distribution this implies that $\sup_{y,y^\prime \in \IT} P_y^\IT [\ov{Y}_t = y^\prime] \le c(\ve)\,N^{-d}$, for $t \ge \ve N^2$. Hence the last term of (\ref{4.37}) is smaller than
\begin{equation*}
c(\ve)\,N^{-d} \,E_z^\IZ\Big[ \dis\int^{\ov{T}_{\wt{I}_\ell}}_0 1\{\ov{Z}_t = z^\prime\}\,dt\Big] \le c (\ve) \;\dis\frac{h_N}{N^d}\,,
\end{equation*}

\medskip\n
with a calculation analogous to that in the last line of (\ref{4.27}). Letting $N$ go to infinity, we see that the expression in (\ref{4.34}) is smaller than 
\begin{equation*}
\underset{N}{\overline{\lim}} \;P_0^{\IZ^{d+1}} \Big[\sup\limits_{0 \le t \le \ve N^2} |\ov{X}_t |_\infty \ge \dis\frac{N}{20M}\Big]\,,
\end{equation*}

\n
which tends to $0$ as $\ve$ goes to $0$, as a result of the invariance principle. This proves (\ref{4.34}) and (\ref{4.31}) follows as well.

\medskip
We are now reduced to the consideration of $\lim_N {\rm cap}_{\wt{C}_\ell}(C_\ell)$. We define
\begin{equation}\label{4.38}
\mbox{$V_{1,\ell},\dots,V_{r_\ell(N),\ell}$ the connected components of $\wt{C}_\ell$, for $1 \le \ell \le L$}\,.
\end{equation}

\medskip\n
The numbers $r_\ell(N)$ in view of the definition (\ref{4.29}) remain bounded by $M$. Each $V_{r,\ell}$ is the union of at most $M$ intersecting closed $|\cdot |_\infty$-balls of $|\cdot |_\infty$-diameter $\frac{2N}{10M} + 1 \le \frac{3N}{10M}$, for large $N$, and thus has diameter at most $\frac{3}{10} \,N$. So for large $N$, each $V_{r,\ell}$ can be identified with some connected set $V^\prime_{r,\ell} \subseteq \IZ^{d+1}$, such that the restriction of the canonical projection $\pi_E$ to $V^\prime_{r,\ell}$ is a bijection onto $V_{r,\ell}$. We write $x^\prime_i$ for the point in $V^\prime_{r,\ell}$ corresponding to $x_i$, when $i \in \cI_\ell$ and $x_i \in V_{r,\ell}$. Note that for $i \in \cI_\ell$ and $x_i \in V_{r,\ell}$, in view of (\ref{4.29}):
\begin{equation}\label{4.39}
B\Big(x^\prime_i, \;\dis\frac{N}{10M}\Big) \subseteq V^\prime_{r,\ell}\,.
\end{equation}

\n
With (\ref{1.6}), (\ref{1.7}) and the corresponding statements when $E$ replaces $\IZ^{d+1}$, we find:
\begin{equation}\label{4.40}
\begin{split}
{\rm cap}_{\wt{C}_\ell}(C_\ell) & = \dsl_{1 \le r \le r_\ell} {\rm cap}_{V_{r,\ell}}\Big(\bigcup\limits_{x_i \in V_{r,\ell}}(x_i + K_i)\Big)
\\[1ex]
& = \dsl_{1 \le r \le r_\ell} {\rm cap}_{V^\prime_{r,\ell}}\Big(\bigcup\limits_{x_i \in V_{r,\ell}}(x^\prime_i + K_i)\Big)\,.
\end{split}
\end{equation}

\medskip\n
Similar bounds to (\ref{4.32}), (\ref{4.33}) where $V^\prime_{r,\ell}$ plays the role of $\wt{C}_\ell$ and $\IZ^{d+1}$ the role of $\wt{B}_\ell$ readily yield that
\begin{equation}\label{4.41}
\lim\limits_N \;{\rm cap}_{\wt{C}_\ell}(C_\ell) - \dsl_{1 \le r \le r_\ell} {\rm cap} \Big(\bigcup\limits_{x_i \in V_{r,\ell}} (x^\prime_i + K_i)\Big) = 0, \;\mbox{for} \; 1 \le \ell \le L\,.
\end{equation}

\n
Moreover in view of (\ref{1.6}), (\ref{1.7}) we have:
\begin{equation}\label{4.42}
\begin{array}{l}
\underset{N}{\overline{\lim}} \;\dsl_{1 \le r \le r_\ell} \Big(\dsl_{x_i \in V_{r,\ell}} {\rm cap}(x^\prime_i + K_i) - {\rm cap} \Big(\bigcup\limits_{x_i \in V_{r,\ell}}(x^\prime_i + K_i)\Big)\Big) \le
\\[1ex]
\underset{N}{\overline{\lim}} \; |C_\ell|^2 \;\sup\{P_x^{\IZ^{d+1}} [H_{x^\prime} < \infty]; \;|x - x^\prime|_\infty \ge \inf\limits_{1 \le i \not= j \le M} d(x_i + K_i, x_j + K_j)\} = 0\,,
\end{array}
\end{equation}

\medskip\n
with the notation $d(\cdot,\cdot)$ from the beginning of Section 1, using (\ref{0.5}) and standard estimates on the walk on $\IZ^{d+1}$, $d + 1 \ge 3$. Note also that the expression inside the limsup in (\ref{4.42}) is non-negative, cf.~(\ref{1.6}), (\ref{1.7}), thus tends to zero with $N$. With translation invariance ${\rm cap}(x^\prime_i + K_i) = {\rm cap}(K_i)$, for $1 \le i \le M$, and coming back to (\ref{4.31}) and (\ref{4.41}) we obtain (\ref{4.21}).
\end{proof}

We will now conclude the proof of Theorem \ref{theo0.1}. With (\ref{4.19}) we see that for large $N$, $1 \le \ell \le L$, $z_1 \in \partial_{\rm int} I_\ell$, $z_2 \in \partial \wt{I}_\ell$,
\begin{equation}\label{4.43}
P_{z_1,z_2}[H_{C_\ell} < T_{\wt{B}_\ell}] \le c\; \dis\frac{h_N}{N^d} \;\dsl^M_{i=1} \,|K_i| \stackrel{\rm def}{=} \delta_N \underset{N \r \infty}{\longrightarrow} 0\,.
\end{equation}

\n
Using the inequality $0 \le e^{-u} - 1 + u \le u^2$, for $u \ge 0$, we see that for $1 \le \ell \le L$,
\begin{equation}\label{4.44}
\begin{array}{l}
\Big| \prod\limits_{2 \le k \le k_*} \big(1 - 1\{Z_{R_k} \in I_\ell\}\,P_{Z_{R_k},Z_{D_k}} [H_{C_\ell} < T_{\wt{B}_\ell}]\big) 
\\[1ex]
- \exp\Big\{- \dsl_{2 \le k \le k_*} 1\{Z_{R_k} \in I_\ell\} \,P_{Z_{R_k},Z_{D_k}} [H_{C_\ell} < T_{\wt{B}_\ell}]\Big\}\Big| \le
\\[1ex]
\delta_N \;\dsl_{2 \le k \le k_*} 1\{Z_{R_k} \in I_\ell\}  \dis\frac{h_N}{N^d} \;c \dsl^M_{i=1} \,|K_i| \,,
\end{array}
\end{equation}

\medskip\n
and with (\ref{2.15}) we see that the $P$-expectation of the last expression goes to zero as $N$ goes to infinity. Analogously we have for $1 \le \ell \le L$,
\begin{equation}\label{4.45}
\begin{array}{l}
E \Big[ \Big| \exp \Big\{-\dsl_{2 \le k \le k_*} 1\{Z_{R_k} \in I_\ell\} \,P_{Z_{R_k},Z_{D_k}} [H_{C_\ell} < T_{\wt{B}_\ell}]\Big\} \;-
\\[1ex]
\qquad  \exp\Big\{ - (d+1) \;\dis\frac{h_N}{N^d} \;\dsl_{i \in \cI_\ell} \,{\rm cap}(K_i) \,\dsl_{2 \le k \le k_*} 1\{Z_{R_k} \in I_\ell\}\Big\}\Big|\Big] \le 
\\[1ex]
E\Big[  \dsl_{2 \le k \le k_*} 1\{Z_{R_k} \in I_\ell\} \,\Big|P_{Z_{R_k},Z_{D_k}} [H_{C_\ell} < T_{\wt{B}_\ell}] - (d+1) \, \dis\frac{h_N}{N^d} \dsl_{i \in \cI_\ell} \,{\rm cap}(K_i) \Big|\Big]
\\[3ex]
\stackrel{(\ref{4.19}),(\ref{4.20}),(\ref{4.30})}{\le} \Big\{c \;\dis\frac{d_N}{h_N} \Big(\dsl_{i \in \cI_\ell} \, |K_i|\Big) + \big| {\rm cap}_{\wt{B}_\ell}(C_\ell) - \dsl_{i \in \cI_\ell} \,{\rm cap}(K_i)\big| \Big\} 
\\[2ex]
E\Big[  \dsl_{2 \le k \le k_*} 1\{Z_{R_k} \in I_\ell\} \,(d+1) \,  \dis\frac{h_N}{N^d}\Big]\,,
\end{array}
\end{equation}

\medskip\n
which tends to zero thanks to (\ref{2.15}), (\ref{3.5}) ii) and (\ref{4.21}). Coming back to the right-hand side of (\ref{4.17}), we see with (\ref{4.44}), (\ref{4.45}) that $A_N$ is limit equivalent to
\begin{equation}\label{4.46}
E\Big[ \exp\Big\{ - \dsl^L_{\ell = 1} \;(d+1) \;\ \dis\frac{h_N}{N^d} \;\dsl_{i \in \cI_\ell} \,({\rm cap} (K_i) + \lambda_i) \,\dsl_{1 \le k \le k_*} 1\{Z_{R_k} \in I_\ell\} \Big\}\Big] \,.
\end{equation}

\n
Consider $\wh{Z}_\point$ from (\ref{1.5}) and define the $(\cF_{\tau_k})_{k \ge 0}$-stopping times $\wh{R}_m, \wh{D}_m, m \ge 1$, which are the successive returns to $C$, cf.~(\ref{2.5}), and departures from $O$ of $\wh{Z}_\point$. Comparing with (\ref{2.8}), we then see that
\begin{equation}\label{4.47}
\mbox{for $k \ge 1$, $\wh{Z}_{\wh{R}_k} = Z_{R_k}, \,\wh{Z}_{\wh{D}_k} = Z_{D_k}$}\,.
\end{equation}

\medskip\n
Moreover by inspection of (\ref{2.10}) and (\ref{2.12}) we see that with hopefully obvious notation:
\begin{equation}\label{4.48}
\begin{array}{l}
t_N(\gamma = (d+1)^{-1}) = (d+1) \,t_N(\gamma = 1), \;\mbox{and}\;
\\[1ex]
 k_*(\gamma = (d+1)^{-1} , \rho = \alpha) = k_* \big(\gamma = 1, \rho = \mbox{\f $\dis\frac{\alpha}{d+1}$}\big)\,.
\end{array}
\end{equation}

\n
Introducing the local time of $\wh{Z}_\point$:
\begin{equation}\label{4.49}
\wh{L}^z_k = \dsl_{0 \le m < k} 1\{\wh{Z}_m = z\}, \;\mbox{for}\; k \ge 0, z \in \IZ\,,
\end{equation}

\medskip\n
it now follows from the observation below (\ref{1.5}) and (\ref{2.14}), (\ref{2.16}) applied with $\gamma = 1$ and $\rho = \frac{\alpha}{d + 1}$, cf.~(\ref{2.12}), that for $1 \le \ell \le L$,
\begin{equation}\label{4.50}
\lim\limits_N \;\sup\limits_{z \in I_\ell} \;E\Big[ \Big| \dis\frac{\wh{L}^z_{[\frac{\alpha}{d+1}\,N^{2d}]}}{N^d} -  \dis\frac{h_N}{N^d} \dsl_{1 \le k \le k_*} 1\{Z_{R_k} \in I_\ell\}  \Big| \wedge 1\Big] = 0\,,
\end{equation}

\n
where we also made use of (\ref{4.47}). With (\ref{4.46}) we thus see that $A_N$ is limit equivalent to
\begin{equation}\label{4.51}
E\Big[ \exp\Big\{ - \dsl^L_{\ell = 1} \,(d+1) \,\dsl_{i \in \cI_\ell} \,({\rm cap}(K_i) + \lambda_i) \, \dis\frac{\wh{L}^{z_i}_{[\frac{\alpha}{d+1}\,N^{2d}]}}{N^d} \Big\}\Big]\,.
\end{equation}

\n
From (1.20) of \cite{CsakReve83}, we can construct a coupling of the simple random walk on $\IZ$ with Brownian motion on $\IR$ in such a fashion that for all $\delta > 0$,
\begin{equation}\label{4.52}
\sup\limits_{z \in \IZ} \,|\wh{L}^z_k - L(z,k)| / k^{\frac{1}{4} + \delta} \underset{k \r \infty}{\longrightarrow} 0,\; \mbox{a.s.},
\end{equation}

\medskip\n
with $L(\cdot,\cdot)$, see below (\ref{0.4}), standing for a jointly continuous version of the local time of the canonical Brownian motion. As a result we see that in the notation introduced below (\ref{0.4}), $A_N$ is limit equivalent to
\begin{equation}\label{4.53}
\begin{array}{l}
E^W \Big[ \exp\Big\{ - \dsl^M_{i=1} \,({\rm cap}(K_i) + \lambda_i)\; \dis\frac{(d+1)}{N^d} \; L \, \Big(z_i, \Big[\mbox{\f $\dis\frac{\alpha}{d+1}$} \;N^{2d}\Big]\Big)\Big\}\Big] \stackrel{\rm scaling}{=}
\\[2ex]
E^W \Big[ \exp\Big\{ - \dsl^M_{i=1} \,({\rm cap}(K_i) + \lambda_i) (d+1) \,L \,\Big(\dis\frac{z_i}{N^d}, \dis\frac{1}{N^{2d}} \,\Big[\mbox{\f $\dis\frac{\alpha}{d+1}$} \;N^{2d}\Big]\Big)\Big\}\Big] 
\end{array}
\end{equation}


\medskip\n
which by dominated convergence and (\ref{0.6}) converges to 
\begin{equation*}
E^W \Big[ \exp\Big\{ - \dsl^M_{i=1} \,({\rm cap}(K_i) + \lambda_i) (d+1) L\Big(v_i, \mbox{\f $\dis\frac{\alpha}{d+1}$}\Big)\Big\}\Big] \,.
\end{equation*}
This concludes the proof of (\ref{4.1}) and of Theorem \ref{theo0.1}. \hfill $\square$

\begin{remark}\label{rem4.3} \rm The methods developed in the proof of Theorem \ref{theo0.1} have a certain robustness and can be adapted to handle variations on the set-up discussed here. But clearly the fact that the times $T_N$ and the related local times involved in the statement of Theorem \ref{theo0.1} can be well approximated  in terms of the coarse-grained information on the vertical component of the walk encaptioned in the $R_k$, $D_k$, $Z_{R_k}, Z_{D_k}, k \ge 1$, see Proposition \ref{prop2.1}, plays an important role. The results presented here leave open a number of natural questions. For instance in the light of \cite{DembSzni06}, \cite{DembSzni07}, \cite{Szni07}, what happens when one chooses $T_N$ to be the disconnection time of the cylinder (and possibly one of the $x_i$ random and equal to $X_{T_N}$)? Another natural problem in the light of \cite{Szni08a} is to investigate what happens when in place of (\ref{0.1}) the cylinders are of the form $G_N \times \IZ$, with $G_N$ a sequence of finite graphs with uniformly bounded degree and cardinality tending to infinity, and the walk is run up to time of order $|G_N|^2$. Here one should for instance consider situations where in the neighborhood of some point $G_N$ tend to look like an infinite graph $G_\infty$, and the walk on $G_\infty \times \IZ$ is transient. These are just a few examples of the questions raised by the present work. \hfill $\square$
\end{remark}

\end{document}